\documentclass[12pt]{article}

\usepackage[a4paper,margin=2.5cm, centering]{geometry}

\usepackage{changepage}

\usepackage{titlesec}

\titleformat{\section}[hang]%
{\bfseries\large}{\thesection.}{1ex}{}%

\titleformat{\subsection}[hang]%
{\bfseries}{\thesubsection}{1ex}{}%

\usepackage{amsthm}
\usepackage{fancyhdr}

\theoremstyle{plain}
\newtheorem{theorem}{Theorem}[section]
\newtheorem{lemma}[theorem]{Lemma}
\newtheorem{proposition}[theorem]{Proposition}
\newtheorem{corollary}[theorem]{Corollary}
\newtheorem{definition}[theorem]{Definition} 

\theoremstyle{definition}
\newtheorem{example}[theorem]{Example}
\newtheorem{remark}[theorem]{Remark}

\usepackage{latexsym}
\usepackage{graphicx}
\usepackage{amsmath}
\usepackage{amssymb}
\usepackage{graphicx}
\usepackage[mathcal]{eucal} 
\usepackage{mathrsfs}
\usepackage{tikz}\usetikzlibrary{calc}
\usepackage{bbm}
\usepackage{url}
\usepackage{subdepth}
\usepackage{hyperref}

\usepackage{ifpdf}
\ifpdf
\usepackage[all,pdf,cmtip]{xy} 
\else
\input xy
\xyoption{all}
\xyoption{v2}
\xyoption{cmtip}
\fi

\newcommand{\RR}{\mathbb{R}}

\newcommand{\LL}{\mathbb{L}}

\newcommand{\cC}{\mathcal{C}}
\newcommand{\cD}{\mathcal{D}}
\newcommand{\cF}{\mathcal{F}}
\newcommand{\cM}{\mathcal{M}}
\newcommand{\cX}{\mathcal{X}}
\newcommand{\cY}{\mathcal{Y}}
\newcommand{\cZ}{\mathcal{Z}}

\newcommand{\vC}{\check{C}}
\newcommand{\cech}[1]{\vC(#1)}

\newcommand{\st}{\rightrightarrows} 
\newcommand{\into}{\hookrightarrow} 
\newcommand{\gento}{-\!\!\!\mapsto}

\DeclareMathOperator{\id}{id}
\DeclareMathOperator{\pr}{pr}
\DeclareMathOperator{\ev}{ev}
\DeclareMathOperator{\disc}{disc}
\DeclareMathOperator{\Lines}{Lines}

\DeclareMathOperator{\cHom}{\underline{\mathcal{H}\hspace{-0.1ex}om}}
\DeclareMathOperator{\Map}{Map}
\newcommand{\inhom}[2]{{#2}^{#1}}

\DeclareMathOperator{\Subm}{Subm}
\DeclareMathOperator{\Subd}{Subd}

\DeclareMathOperator{\Cart}{\bf{Cart}}
\DeclareMathOperator{\Gpd}{\bf{Gpd}}
\DeclareMathOperator{\Set}{\bf{Set}}
\DeclareMathOperator{\Stack}{\bf{Stack}}

\DeclareMathOperator{\String}{String}
\DeclareMathOperator{\act}{act}

\newcommand{\bigsubset}[1]
{ \rotatebox{#1}{$\subset$}}

\def\centerarc[#1](#2)(#3:#4:#5)
{ \draw[#1] ($(#2)+({#5*cos(#3)},{#5*sin(#3)})$) arc (#3:#4:#5) } 

\title{Smooth loop stacks of differentiable stacks and gerbes}

\author{David Michael Roberts
and Raymond F. Vozzo}
\date{25 November 2016}

\begin{document}

\maketitle
\mbox{ }

\vskip 25pt
\begin{adjustwidth}{0.5cm}{0.5cm}
{\small
{\bf R\'esum\'e.} Nous d\'efinissons un groupo\"ide de Fr\'echet-Lie $\Map(S^1,X)$ d'ana-foncteurs du cercle vers un groupo\"ide de Lie $X$.  
Ceci fournit une pr\'esentation du Hom-champ $\cHom(S^1,\cX)$, o\`u $\cX$ est le champ diff\'erentiable associ\'e \`a $X$.
Nous appliquons cette construction au groupo\"ide de Lie  sous-jacent au `gerbe fibr\'e' d'une vari\'et\'e diff\'erentiable $M$; le r\'esultat est un gerbe fibr\'e au-dessus de l'espace des lacets $LM$ de $M$.\\
{\bf Abstract.} We define a Fr\'echet--Lie groupoid $\Map(S^1,X)$ of anafunctors from the circle into a Lie groupoid $X$. 
This provides a presentation of the Hom-stack $\cHom(S^1,\cX)$, where $\cX$ is the differentiable stack associated to $X$. 
We apply this construction to the Lie groupoid underlying a bundle gerbe on a manifold $M$; the result is a bundle gerbe on the loop space $LM$ of $M$.
}
\end{adjustwidth}


\section{Introduction}\label{sec:intro}

The notion of smooth loop space of a manifold is useful in a variety of areas of geometry, while at the same time being just outside the usual sphere of study, namely finite-dimensional manifolds. 
While it is naturally a topological space, it carries a very well-behaved smooth structure as an infinite-dimensional manifold.
While recent progress on generalised smooth spaces means that any mapping space, in particular a loop space, is easily a smooth space, and the general study of smooth spaces is advancing in leaps and bounds (see e.g.~the lengthy book \cite{IZ} on diffeological spaces), the fact that the loop space is a manifold with well-understood charts is extremely useful.

The area of geometry has in recent years expanded to include what is becoming known as `higher geometry', where, loosely speaking, the geometric objects of study have a categorical or higher categorical aspect.
One example of such objects are differentiable, or \emph{Lie}, groupoids, which are known \cite{Pronk_96} to be incarnations of \emph{differentiable stacks}: stacks that look locally like manifolds, but with internal symmetries captured by Lie groupoids.
A rather well-known simple case is that of \emph{orbifolds}.
Other examples that are still stacks on manifolds but which are still akin to Lie groupoids, are groupoids built from infinite-dimensional manifolds, or from smooth spaces. 
Clearly these objects can become locally less well-behaved as one becomes more general; an arbitrary diffeological space, for instance, may have rather terrible topological and homotopical properties.

The construction that this short paper wishes to address is that of the loop stack of a differentiable stack. 
This was introduced in the special case of orbifolds in \cite{Lupercio-Uribe_02}, and then considered in full generality for the purposes of studying string topology in \cite{Behrend_et_al_12}.
All of these are special cases of the loop stack of the underlying \emph{topological stack}, a special case of the topological mapping stack studied in \cite{Noohi_10}.\footnote{The paper \cite{Carchedi_12} considers the more general problem of a cartesian closed bicategory of stacks, whereas \cite{Noohi_10} considers the special case with compactness conditions on the domain.}
In other words, the end result is only a \emph{topological} stack, rather than a differentiable stack.

One (quite reasonable) approach is to consider if we can find a loop stack on manifolds that arises from a diffeological groupoid (i.e.~a diffeological stack).
This is not too difficult, and the parts of our construction that do not require special handling due to the nature of manifolds are performed for diffeological stacks.
The novelty here is that this construction can be lifted so that it becomes a stack arising from what we call a \emph{Fr\'echet--Lie} groupoid: a groupoid in the category of Fr\'echet manifolds. 
This is the optimal result, since the construction applied to a manifold returns (a stack equivalent to) the usual loop space of that manifold, which is an infinite-dimensional Fr\'echet manifold in general. 
This is in contrast to the case of algebraic Hom-stacks, for  instance \cite{Olsson_06}, which are again algebraic stacks.

The benefits of having a smooth version of the loop stack is that one can start to do actual geometry on it, rather than just topological constructions (such as the string topology in \cite{Behrend_et_al_12}). 
Moreover, while one can perform smooth geometric constructions on diffeological spaces, as \emph{spaces}, unlike manifolds there is little control over the local structures.
So, for instance, our construction provides a \emph{smoothly paracompact} groupoid, admitting partitions of unity on object and arrow manifolds.

Another example, which was the original impetus for this article, are the loop stacks of \emph{bundle gerbes}.
Bundle gerbes over manifolds are higher geometric objects analogous to line bundles, and as such can support structures analogous to connections.
One can form the construction given below to the groupoid underlying a bundle gerbe and then the resulting groupoid is in fact still a gerbe, now over a loop space, and this should again carry a connective structure of the appropriate sort.
Of particular interest is the bundle gerbe underlying the \emph{String 2-group}, which will be the subject of future work.

\begin{figure}[!b]
\[
\centerline{
	\xymatrix{
		\cM = \{\text{manifolds}\} \ar@{}[r]|-{\bigsubset{0}} \ar@{}[d]|-{\bigsubset{-90}} 
		& \{\text{Fr\'echet manifolds}\} \ar@{}[r]|-{\bigsubset{0}} \ar@{}[d]|-{\bigsubset{-90}}  
		& \{\text{diffeological spaces}\} = \cD \ar@{}[d]|-{\bigsubset{-90}} 
		\\
		\{\text{Lie groupoids}\} \ar@{}[r]|-{\bigsubset{0}} \ar@{->>}[d]  
		& \{\text{Fr\'echet Lie groupoids}\} \ar@{}[r]|-{\bigsubset{0}} \ar@{->>}[d] 
		& \{\text{diffeological groupoids}\} \ar@{->>}[d] 
		\\
		\{\text{Differentiable stacks}\} \ar@{}[r]|-{\bigsubset{0}} \ar@{}[d]|-{\bigsubset{-90}} 
		& \{\text{Fr\'echet differentiable stacks}\} \ar@{}[d]|-{\bigsubset{-90}} 
		& \{\text{Diffeological stacks}\} \ar@{}[d]|-{\bigsubset{-90}} 
		\\
		\Stack_\cM \ar@{=}[r]& \Stack_\cM \ar[r]^{\simeq}& \Stack_\cD 
	}
	}
\]
\caption{Categories and 2-categories appearing in the paper}\label{fig:categories}
\end{figure}
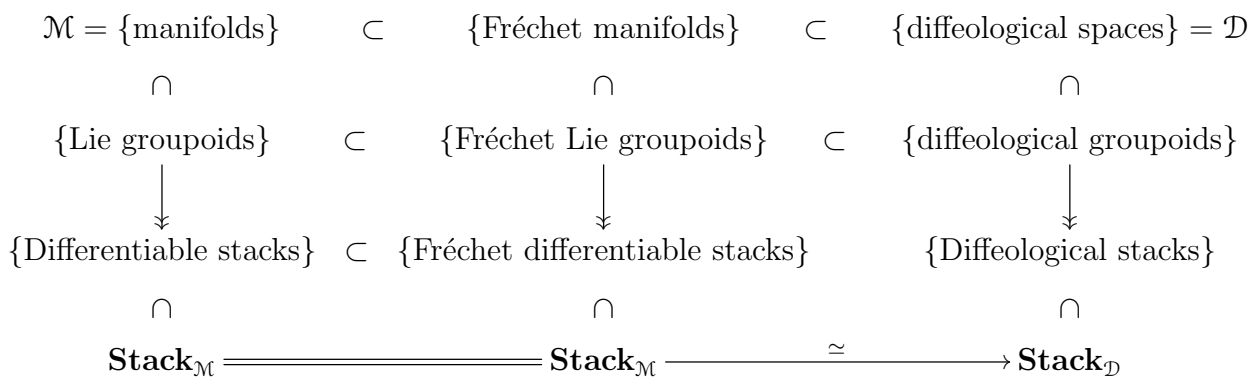

We consider in this paper various categories of smooth objects, groupoids in those categories and corresponding smooth stacks. 
Figure \ref{fig:categories} summarises these, as well as the relations between them. 
The first row consists of categories, the remaining rows consist of 2-categories, and the inclusions denote \emph{full} subcategories and sub-2-categories. 
The vertical arrows of type $\xymatrix{\ar@{->>}[r]&}$ denote surjective-on-objects 2-functors. 
We use $\Stack_X$ to denote the 2-category of stacks of groupoids on the site $X$.


The paper outline is as follows:

\begin{itemize}

	\item Section 2---Gives background on sites, internal groupoids, anafunctors (a type of generalised morphism between internal groupoids) and stacks presented by groupoids internal to the base site.
	\item Section 3---We construct a diffeological groupoid $\Map(S^1,X)$ of anafunctors and transformations.
	\item Section 4---Proves that $\Map(S^1,X)$ is indeed a presentation over the site of diffeological spacs, making $\cHom(S^1,\cX)$ a diffeological stack.
	\item Section 5---We show that the construction of $\Map(S^1,X)$ actually lands in the sub-2-category of Fr\'echet--Lie groupoids, and that this gives a (weak) presentation of $\cHom(S^1,\cX)$. This is our first main result.
	\item Section 6---Gives a treatment of the theory of gerbes on the site of manifolds presented by (Fr\'echet--)Lie groupoids, including establishing stability of various properties under forming the mapping groupoid.
	\item Section 7---We prove our second main result, namely that given a bundle gerbe (a special sort of abelian gerbe), the mapping groupoid is again a bundle gerbe.
\end{itemize}

{\bf Acknowledgements.} This research was supported Australian Research Council's Discovery Projects DP120100106 and DP130102578.
This project was born in `Coffee Spot': thanks to the staff for uninterrupted and secluded working time. 
A big thanks to Andrew Stacey for writing the paper \cite{Stacey_13}, in order to prove theorem \ref{Staceys_thm}, after discussions with the first named author; this theorem was crucial to the success of the current paper. 
Thanks also to Alexander Schmeding for side discussions about possible extensions to the infinite-dimensional setting. 
The authors thank the anonymous referee for their careful reading and helpful suggestions, which helped us find a small error in the original version of Lemma~\ref{lemma:LR_actions_agree}.

\section{Background and preliminaries}\label{sec:background}

\subsection{Sites}\label{subsec:sites}

We will be interested in stacks over sites where the Grothendieck topology, arises from a \emph{coverage} (see e.g.~\cite[Section C.2.1]{Elephant}), rather than the more familiar data of a pretopology.
In this paper we will work only with a coverage and not the Grothendieck topology generated by it.

\begin{definition}\label{def:site}

	Let $\cC$ be a category. A \emph{coverage} $J$ on $\cC$ is a collection $J(x)$, for each object $x$, of families of arrows $\{u_i \to x\mid i \in I\}$ (called \emph{covering families}) with the property that for each covering family $\{u_i \to x\mid i\in I\}\in J(x)$ and $f \colon y\to x$ there is a covering family $\{v_k \to y\mid k\in K\} \in J(y)$ such that for all $k$ there is an $i\in I$ and lift as shown
		\[
			\xymatrix{
				v_k \ar@{-->}[r] \ar[d] & u_i \ar[d]\\
				y \ar[r]_f & x
			}
		\]
	A \emph{site} $(\cC,J)$ is then a category $\cC$ equipped with a coverage $J$, and sites with the same underlying category are \emph{equivalent} if their coverages generate the same sieves.

\end{definition}

It will be the case that the coverages we consider satisfy the saturation condition that composites of coverages are again coverages, but not still not necessarily that pullbacks of covering families are covering families.

If we have a pair of covering families $\mathcal{U} = \{u_i \to x\mid i \in I\}$ and $\mathcal{V} = \{v_j \to x\mid j \in J\}$ then we say $\mathcal{V}$ \emph{refines} $\mathcal{U}$ if for every $j\in J$ there is an $i\in I$ and a lift of $v_j\to x$ through $u_i$. 
We can say that a coverage $J_1$ refines the coverage $J_2$ if every covering family in $J_1$ refines a covering family in $J_2$. If $J_1$ refines $J_2$ and $J_2$ refines $J_1$ then they give rise to equivalent sites.

A coverage is called a \emph{singleton} coverage if all covering families consist of single maps, in which case covering families will be referred to as \emph{covering maps}. 
An example of a singleton coverage is a class of maps containing identity arrows, closed under composition and pullback along arbitrary maps; such a class will be called a \emph{singleton pretopology}

A \emph{superextensive} coverage (on an extensive category, see \cite{Carboni-Lack-Walters_93}) is one that is generated by a singleton coverage and the coverage where covering families are inclusions of summands $\{u_i \to \coprod_{i\in I} u_i\mid i\in I\}$. 
For all intents and purposes, a superextensive coverage $J$ can be reduced to considering just the singleton coverage $\amalg J$ it gives rise to: $\amalg J$-covering maps are of the form $\coprod_i u_i \to x$, for $\{u_i \to x\mid i\in I\}$ a covering family in the original superextensive coverage.
We shall abuse terminology slightly and say that a superextensive coverage $J_1$ and another singleton coverage $J_2$ give rise to equivalent sites when the \emph{singleton coverage associated to $J_1$} is equivalent to $J_2$.
We shall also abuse notation and refer to a covering map in $\amalg J$ as being in $J$ when no confusion shall arise.

A site is called \emph{subcanonical} if all representable presheaves are in fact sheaves.
For a singleton coverage this is implied by all covering maps being regular epimorphisms, and for a subcanonical superextensive coverage $J$, the singleton coverage $\amalg J$ is subcanonical. 
In fact all of the coverages we consider in this paper will be subcanonical.

We will need the following examples over the course of the paper.

\begin{example}\label{eg:Cart_site}

	Consider the category $\Cart$ with objects $\mathbb{R}^n$ for $n=0,1,2,\ldots$ and $\Cart(\mathbb{R}^n,\mathbb{R}^m) = C^\infty(\mathbb{R}^n,\mathbb{R}^m)$. This has a coverage where a covering family $\{\phi_i\colon \mathbb{R}^n \into \mathbb{R}^n\mid i\in I\}$ is an open cover in the usual sense.

\end{example}

For the purposes of the current paper, we consider manifolds to be finite dimensional unless otherwise specified.

\begin{example}\label{eg:Man_site}
	
	The category $\mathcal{M}$ of smooth manifolds has the following coverages:
	\begin{itemize}
		\item the coverage ${O}$ of open covers in the usual sense;
		\item the coverage $C$, where covering families $C(X)$ are covers of $X$ by regular closed compact neighbourhoods, such that the interiors also cover; 
		\item the singleton pretopology $\Subm$ where covering maps are surjective submersions.
		
	\end{itemize}
\end{example}

All these coverages give equivalent sites, the first two because manifolds are locally compact and regular\footnote{In fact there is a coverage on the category of locally compact spaces consisting of compact neighbourhoods, and a coverage on the category of regular spaces consisting of closed neighbourhoods.} and the first and last because surjective submersions have local sections. 
The first two coverages are superextensive, and we will be considering their associated singleton coverages.

Recall that a (smooth) \emph{Fr\'echet} manifold is a smooth manifold locally modelled on Fr\'echet spaces (a good reference is \cite{Hamilton_82}). 
The definition does not assume second-countability, so that the category of Fr\'echet manifolds admits small coproducts.
A submersion between Fr\'echet manifolds is a map for which there are charts on which the map looks locally like a projection out of a direct sum: $V\oplus W\to V$ (it is not enough to ask that this is surjective, or even split surjective, on tangent spaces).

\begin{example}
	
	The category $\mathcal{F}$ of Fr\'echet manifolds has a coverage given by open covers, and also a singleton pretopology given by surjective submersions. 
	The first is superextensive, and these give rise to equivalent sites.

\end{example}

Our last example needs some preliminaries. The following definition is quite different to that which appears in the original article \cite{Souriau_80}, but is in fact equivalent by work of Baez--Hoffnung \cite{Baez-Hoffnung_11}. 
An extensive reference is the book \cite{IZ}.

\begin{definition}

	A \emph{diffeological space} is a sheaf $X$ on $\Cart$ that is a subsheaf of $\mathbb{R}^n \mapsto \Set(\mathbb{R}^n,\underline{X})$, where $\underline{X} = X(\mathbb{R}^0)$ is the \emph{set of points of $X$}. 
	A smooth map of diffeological spaces is just a map between the underlying sheaves.
	We denote the category of diffeological spaces by $\mathcal{D}$.

\end{definition}

We can think of cartesian spaces $\mathbb{R}^n$ as diffeological spaces via the Yoneda embedding, and for $X$ a diffeological space, the elements of $X(\mathbb{R}^n)$ as maps $\mathbb{R}^n \to X$ in $\mathcal{D}$.
The category of diffeological spaces is a Grothendieck quasitopos \cite{Baez-Hoffnung_11}, in particular is complete, cocomplete, extensive and cartesian closed.

A map $X\to Y$ of diffeological spaces is a \emph{subduction} if for every $f\colon \mathbb{R}^n \to Y$ there is a covering family $\phi_i\colon \mathbb{R}^n \into \mathbb{R}^n$ such that each map $f\circ \phi_i\colon \mathbb{R}^n \to Y$ lifts to $X$.
Note that there are fully faithful inclusions $\mathcal{M} \into \mathcal{F} \into \mathcal{D}$. 
Surjective submersions of manifolds and also of Fr\'echet manifolds are subductions.

\begin{example}

	The category of diffeological spaces has a singleton pretopology $\Subd$ given by subductions.

\end{example}

The following facts about subductions will be useful.

\begin{itemize}

	\item Every subduction $A\to B$ is refined by a subduction with domain a coproduct of Euclidean spaces;

	\item Every subduction $A\to M$ with $M$ a manifold is refined by an open cover of $M$.

\end{itemize}

The astute reader will have noticed that almost all of the examples are in fact pretopologies or singleton pretopologies.
The important fact is that we need to use the singleton coverage $\amalg C$ which is \emph{not} a pretopology, but \emph{refines} a singleton pretopology.

\subsection{Internal groupoids}\label{subsec:gpds_anafunctors}

We will be dealing with internal groupoids that satisfy extra conditions, due to the fact that the ambient categories of manifolds are not finitely complete.
To that end, \emph{Lie groupoids} are groupoids internal to $\cM$ where the source and target maps are submersions, and \emph{Fr\'echet--Lie groupoids} are groupoids internal to $\cF$ where again the source and target maps are submersions of Fr\'echet manifolds.
We will also consider \emph{diffeological} groupoids, which are just groupoids internal to $\cD$; their source and target maps are automatically subductions.

Functors between internal groupoids, be they Lie, Fr\'echet--Lie or diffeological groupoids, will be assumed to be smooth.
The same will be true for natural transformations between such functors.
We denote, for a category $\cC$, the 2-category of groupoids internal to $\cC$ by $\Gpd(\cC)$, with the above caveats for $\cC = \cM,\ \cF$.
Since the inclusions $\cM \into \cF \into \cD$ are full, we have full inclusions of 2-categories $\Gpd(\cM) \into \Gpd(\cF) \into \Gpd(\cD)$.

It is a well-known problem that there  are just not enough morphisms between internal groupoids, in particular Lie groupoids and their cousins.
One approach to this problem is through the use of \emph{internal anafunctors}.
These were introduced in Bartels' thesis \cite{Bartels}, inspired by work of Makkai on foundational issues surrounding the Axiom of Choice in category theory.
We do not need the full theory of internal anafunctors, the basic definitions are enough for the present paper, for the special case where we only consider internal groupoids.
We have also generalised the notion ever so slightly, by using singleton \emph{coverages}; the fragment of the theory we need here does not lose out by considering this more general setting.

\begin{definition}[\cite{Bartels}]
	
	Let $J$ be a singleton coverage on $\cC$ and let $Y$ and $X$ be groupoids in $\cC$.
	An \emph{anafunctor} $Y\gento X$ is a span of internal functors
	\[
		Y \xleftarrow{j} Y' \xrightarrow{f} X
	\]
	where the object component $j_0\colon Y'_0 \to Y_0$ of $j$ is a $J$-cover, and the following square is a pullback
	\[
		\xymatrix{
			Y'_1 \ar[r]^{j_1} \ar[d] & Y_1 \ar[d]\\
			Y'_0\times Y'_0 \ar[r]_{j_0\times j_0} & Y_0 \times Y_0
		}
	\]

\end{definition}

Of primary interest to us is the case when the groupoid $Y$ has no nontrivial arrows, that is, it is just an object of $\cC$, say $M$.
In that case, any functor $j\colon Y' \to M$ satisfying the conditions is determined by the map on objects and the groupoid $Y$ is what is known as a \emph{\v Cech groupoid} of the covering map $j_0$ (or by abuse of notation, of its domain).
If we let $U= Y'_0$, then $Y'_1 = U\times_M U$, and we denote $Y'$ by $\check C(U)$.
Thus any anafunctor from $M$ to an internal groupoid $X$ is of the form $M \xleftarrow{j} \check C(U) \xrightarrow{f} X$.

Assume for the moment that $J$ is a singleton pretopology, so that we have pullbacks of covering maps. 
Given a pair of anafunctors $M \leftarrow \check C(U_1) \xrightarrow{f} X$ and $M \leftarrow \check C(U_2) \xrightarrow{g} X$, we want to define what it means to have a transformation between them.
Let $U_{12} = U_1 \times_M U_2$. Then a transformation is a diagram
\[
	\xymatrix{
		& \check C(U_1) \ar[dl] \ar[dr]^f_(.6){\ }="s" \\
		M & \check C(U_{12}) \ar[u] \ar[d] & X\; ,\\
		& \check C(U_2) \ar[ul] \ar[ur]_g^(.6){\ }="t"
		\ar@{=>}"s";"t"^\alpha
	}
\]
where the two functors $\check C(U_{12}) \to \check C(U_i)$ are induced by the projections $U_{12} \to U_i$.
The picture one should keep in mind here is a coboundary between $X$-valued \v Cech cocycles that lives over a common refinement.

For a singleton \emph{coverage}, such as the coverage $\amalg C$ of compact neighbourhoods on manifolds, we can define a transformation to be a diagram as above, where instead of considering the pullback, which does not necessarily exist (or if it does, may not be a covering map), one considers a refinement $U_{12}$, equipped with maps to $U_1$ and $U_2$. 
One of the lessons that can be gleaned from \cite{Roberts_16} is that when working with anafunctors nothing is lost by considering a coverage that is cofinal in a pretopology, rather than the pretopology itself (as in \cite{Bartels}).

Using the notion of anafunctor with respect to a pretopology, internal groupoids, anafunctors and transformations form a bicategory \cite{Bartels}. 
We will not use this bicategory structure directly, but it is relied on implicitly to take advantage of Theorem~\ref{thm:anafunctors_are_stack_maps} below.

\subsection{Stacks}\label{subsec:stacks}

We are considering stacks on the category $\cM$ of manifolds using the coverage $O$ of open covers. 
A standard reference is \cite{Behrend-Xu_11}, and we point the reader to the detailed discussion of stacks in section 2.2 therein. 
We give the definition we need and then mention without proof some standard facts.

\begin{definition}

	Let $\cX\colon \cM^{\text{op}} \to \Gpd$ be a weak 2-functor. We say $\cX$ is a \emph{stack} if the following conditions are satisfied for every covering family $\{\phi_i\colon U_i\to M\mid i\in I\}$:
	\begin{enumerate}
		
		\item For any pair of objects $x,y$ of $\cX(M)$ and any family of isomorphisms $\sigma_i\colon x|_{U_i} \to  y|_{U_i}$ in $\cX(U_i)$, $i\in I$, there is a \emph{unique} isomorphism $\sigma\colon x\to y$ in $\cX(M)$ such that $\sigma|_{U_i} = \sigma_i$.

		\item For every family of objects $x_i \in \cX(U_i)$, $i\in I$, and collection of isomorphisms $\sigma_{ij}\colon x_i|_{U_{ij}} \to x_j|_{U_{ij}}$ in $\cX(U_{ij})$, $i,j\in I$ satisfying $\sigma_{jk}\circ\sigma_{ik} = \sigma_{ik}$ in $\cX(U_{ijk})$ (leaving the restrictions implicit), then there is an object $x$ of $\cX(M)$ and isomorphisms $\rho_i\colon x|_{U_i} \to x_i$ for all $i\in I$ such that $\sigma_{ij}\circ \rho_i = \rho_j$ (in $\cX(U_{ij})$) for all $i,j\in I$ (where as usual, we write $U_{ij} = U_i\cap U_j$ and $U_{ijk}=U_i\cap U_j\cap U_k$).
	\end{enumerate}

	If only the first point is satisfied, then we say $\cX$ is a \emph{prestack}.

\end{definition}

A morphism of stacks is given by a transformation of weak 2-functors, and there is a 2-category $\Stack_\cM$ of stacks on $(\cM,O)$.
The relevant points we need are as follows:

\begin{itemize}

	\item Any manifold $M$ gives rise to a stack, also denoted by $M$ (as $O$ is subcanonical). Also, any diffeological space is a stack. 
	The Yoneda embedding ensures that any map of stacks between manifolds or diffeological spaces is just a smooth map in the usual sense.
	A stack equivalent to a manifold is called \emph{representable}.

	\item Any Lie groupoid gives rise to a \emph{prestack}, by sending the groupoid $X$ to the presheaf of groupoids $\cM(-,X)\colon \cM^\text{op} \to \Gpd$, and this prestack can be `stackified'. More generally, any Fr\'echet--Lie or diffeological groupoid gives rise to a prestack and hence a stack.

	\item The 2-category of stacks $\Stack_\cD$ on $(\cD,\Subd)$ is equivalent to $\Stack_\cM$. 
	This follows from a stack version of the ``lemme de comparaison'' \cite[Espos\'e III, Th\'eor\`eme 4.1]{SGA4.1}; see discussion at \cite{Carchedi_MO}.

\end{itemize}

The correct notion of `pullback' for stacks is a \emph{comma object}.\footnote{This is sometimes called a weak pullback, or even just a pullback, in the stack literature. However the definition usually given is clearly that of a comma object.} 
For a cospan $G \xrightarrow{f} H \xleftarrow{g} K$ of groupoids, the comma object $G\downarrow_H K$ (or sometimes $f\downarrow g$) can be computed as the \emph{strict} limit $G\times_H H^\mathbbm{2} \times_H K$ where $H^\mathbbm{2}$ is the arrow groupoid of $H$. 
The comma object of a cospan of stacks is calculated pointwise, that is, $\left(\cX \downarrow_\cZ \cY\right)(M) = \cX(M) \downarrow_{\cZ(M)} \cY(M)$. The comma object fits into a 2-commuting square called a \emph{comma square},
\[
	\xymatrix{
		\cX \downarrow_\cZ \cY\ar[r] \ar[d] & \cY \ar[d]_{\ }="s"\\
		\cX \ar[r]^{\ }="t" & \cZ
		\ar@{=>}"s";"t"
	}
\]
which is universal among such 2-commuting squares.

A stack is said to be \emph{presentable} if it is the stackification of an internal groupoid. In this case, there is extra structure that the stack admits, from which we can recover the groupoid up to weak equivalence \cite{Pronk_96,Behrend-Xu_11}.

First, we say a map of stacks $\cY \to \cX$ is \emph{representable} (resp.~representable by diffeological spaces) if for every manifold $M$ and map $M \to \cX$, the comma object $M\downarrow_\cX \cY$ is representable by a manifold (resp.~a diffeological space).
We can talk about properties of representable maps arising from properties of maps in $\cM$ or $\cD$; if $P$ is a property of maps of manifolds (or diffeological spaces) that is stable under pullback and local on the target in a given coverage $J$, then we say a representable map of stacks $\cY \to \cX$ has property $P$ if for every $M \to \cX$ the projection $M\downarrow_\cX \cY \to M$ has property $P$.

\begin{definition}
	
	A stack $\cX$ on $(\cM,O)$ is \emph{presentable} (resp.~presentable by a diffeological groupoid) if there is a manifold (resp.~diffeological space) $X_0$ and a representable epimorphism $p\colon X_0 \to \cX$ that is a submersion (resp.~subduction).

\end{definition}

It follows from the definition that the comma object $X_1 := X_0 \downarrow_\cX X_0$ is representable (by a manifold or diffeological space), the two projection maps $X_1 \to X_0$ are submersions (or subductions) and $X_1 \st X_0$ is an internal groupoid.
This internal groupoid is said to \emph{present} the stack $\cX$.
Then $\cX$ is the stackification of the prestack arising from this internal groupoid.
Note that this definition also works if we ask for presentability by a Fr\'echet--Lie groupoid: one asks for a representable submersion from a Fr\'echet manifold.

The usual name for a stack presentable by a Lie groupoid is \emph{differentiable stack}, and we will call stacks presentable by diffeological groupoids, \emph{diffeological stacks}.
Stacks presented by a Fr\'echet--Lie groupoid shall be called \emph{Fr\'echet differentiable stacks}.

The main result we need here is the following, and follows from the combination of the general theory of \cite{Pronk_96} and \cite[Theorem 7.2]{Roberts_12} in the case of Lie groupoids, and uses an adaptation of Pronk's argument for the case of diffeological groupoids.

\begin{theorem}\label{thm:anafunctors_are_stack_maps}

	The 2-category of differentiable stacks (resp.\ diffeological stacks) is equivalent to the bicategory of Lie groupoids (resp.\ diffeological groupoids), anafunctors and transformations.

\end{theorem}

What this means in practice is that we can pass between maps between presentable stacks and anafunctors between the presenting groupoids, and we shall use this below.

If we have an epimorphism $p\colon X_0 \to \cX$ from a representable stack $X_0$ such that merely the comma object $X_0 \downarrow_\cX X_0$ is representable and the projections are surjective submersions, then we call $p$ a \emph{weak presentation}.
For certain sites a weak presentation gives a strong presentation: this is true for instance for presentations by diffeological spaces.
This relies on the following lemma adapted from \cite[Lemma~2.2]{Behrend-Xu_11}, which works in the framework of stacks on the site of \emph{not-necessarily-Hausdorff} (finite-dimensional) manifolds.

\begin{lemma}\label{lemma:representable_maps_local_on_target}

	Let $f\colon \cY \to \cX$ be a morphism in $\Stack_\cM$. 
	If $M$ is a diffeological space, $M\to \cX$ an epimorphism of stacks, and the comma object $M\downarrow_{\cX} \cY$ is a diffeological space, then $f$ is representable as a map of stacks considered in the equivalent 2-category $\Stack_\cD$.

\end{lemma}

The analogous result is not true for stacks on the category $\cM$, but it \emph{is} true (following \cite{Behrend-Xu_11}) if we allow ourselves to use possibly non-Hausdorff manifolds.
In practice, one often finds that the stack is weakly presented by a Lie groupoid, which is made up of (Hausdorff) manifolds, which then can be used without reference to non-Hausdorff manifolds.
The same can be said for weak presentations by Fr\'echet--Lie groupoids, an example of which will arise in our main construction.

While it is not always the case that the 2-category of internal groupoids has internal homs, the 2-category of stacks \emph{does} have internal homs, namely for a pair of stacks $\cX,\cY$, there is a stack $\cHom(\cY,\cX)$ and an evaluation map $\cY \times \cHom(\cY,\cX) \to \cX$ with the necessary properties.

\begin{definition}
	
	The \emph{Hom-stack} $\cHom(\cY,\cX)$ is defined by taking the value on the object $M$ to be the groupoid $\Stack_\cM(\cY\times M,\cX)$.

\end{definition}

Thus we have a Hom-stack for any pair of stacks on $\mathcal{M}$.
The case we are interested in is where we have a stack $\cX$ associated to an internal groupoid $X$ in $\cM$ or $\cD$, and the Hom-stack $\cHom(S^1,\cX)$.

\section{Construction of the diffeological loop groupoid}\label{section:diffeological_groupoid}

We will now describe the construction of the loop groupoid of a diffeological groupoid $X$. This will naturally be a groupoid also internal to $\cD$, and we shall show in the next section that it in fact presents the Hom-stack $\cHom(S^1, \cX)$, for $\cX$ the stack associated to $X$.

The objects of the diffeological mapping groupoid are anafunctors $S^1 \gento X$, using the compact neighbourhood coverage $C$ of Example \ref{eg:Man_site}.

As the category $\cD$ of diffeological spaces is cartesian closed and finitely complete, results of Bastiani--Ehresmann \cite{Bastiani-Ehresmann_72} imply that the category $\Gpd(\cD)$ of diffeological groupoids is also cartesian closed and finitely complete. 
Therefore the set $\Gpd(\cD)(\check C(V), X)_0$ of objects of the internal hom---a groupoid---is in fact a diffeological space.
We shall, for the sake of saving space, write $\inhom{\vC(V)}{X} := \Gpd(\cD)(\check C(V),X)_0$.
The category $\cD$ is also cocomplete (in fact extensive) and so we define the object space $\Map(S^1,X)_0$ to be the diffeological space
\[
	\coprod_{V\in C(S^1)} \inhom{\vC(V)}{X}.
\]

\begin{remark}
This mirrors the construction of the topological loop group\-oid as in \cite{Lupercio-Uribe_02, Behrend_et_al_12}, even though for the purposes of diffeological group\-oids it is not necessary to focus only on \emph{compact} neighbourhoods; the diffeological groupoid of functors would exist using ordinary open covers. 
This would even give an equivalent mapping diffeological groupoid in the end.
However, for the Lie groupoids in section \ref{sec:Frechet} we \emph{do} need to use compact covers to get the appropriate Fr\'echet topology on mapping spaces.

\end{remark}

The picture we keep in mind for the elements of the object space is  a sequence like:

\medskip
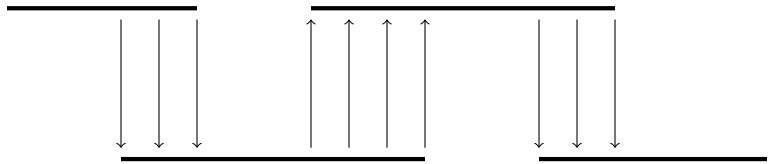
\begin{center}
\begin{tikzpicture}
\pgfmathsetmacro{\top}{2}
\pgfmathsetmacro{\bottom}{0}
\pgfmathsetmacro{\smallgap}{0.15}
\pgfmathsetmacro{\arrowlength}{\top-\bottom - 2*\smallgap}

	\draw[ultra thick] (0,\top) -- (2.5,\top);
		\draw[->] (1.5,\top-\smallgap) -- +(0,-\arrowlength);
		\draw[->] (2,\top-\smallgap) -- +(0,-\arrowlength);
		\draw[->] (2.5,\top-\smallgap) -- +(0,-\arrowlength);
	\draw[ultra thick] (1.5,\bottom) -- (5.5,\bottom);
		\draw[<-] (4,\top-\smallgap) -- +(0,-\arrowlength);
		\draw[<-] (4.5,\top-\smallgap) -- +(0,-\arrowlength);
		\draw[<-] (5,\top-\smallgap) -- +(0,-\arrowlength);
		\draw[<-] (5.5,\top-\smallgap) -- +(0,-\arrowlength);
	\draw[ultra thick] (4,\top) -- (8,\top);
		\draw[->] (7,\top-\smallgap) -- +(0,-\arrowlength);
		\draw[->] (7.5,\top-\smallgap) -- +(0,-\arrowlength);
		\draw[->] (8,\top-\smallgap) -- +(0,-\arrowlength);
	\draw[ultra thick] (7,\bottom) -- (10,\bottom);
	
\end{tikzpicture}
\end{center}
\medskip
where the horizontal lines are paths in $X_0$ and the vertical arrows, varying smoothly,
are given by a path in $X_1$.

Next we move on to the arrow space of $\Map(S^1, X)$.
Recall that a \emph{transformation} $t\colon f\to g$ of anafunctors $f,g\colon S^1 \gento X$ is a diagram
\[
	\xymatrix{
		\check C(V_{12}) \ar[r] \ar[d] & \check C(V_1) \ar[d]^f_(.35){\ }="s"\\
		\check C(V_2) \ar[r]_g^{\ }="t" & X
		\ar@{=>}"s";"t"
	}
\]
where $V_{12}$ is the chosen refinement of $V_1 \times_{S^1} V_2$ as discussed in section \ref{subsec:gpds_anafunctors}. 
Note that $t$ is necessarily a natural isomorphism as $X$ is a groupoid.

For arbitrary $f$ and $g$ with domains $\check C(V_1)$ and $\check C(V_2)$, respectively, the diffeological space of all transformations is
\[
	\inhom{\vC(V_1)}{X} \times_{\inhom{\vC(V_{12})}{X}} \inhom{\vC(V_{12})}{(X^\mathbbm{2})} \times_{\inhom{\vC(V_{12})}{X}} \inhom{\vC(V_2)}{X}
\]
where the two maps 
\[
	\inhom{\vC(V_i)}{X} \to \inhom{\vC(V_{12})}{X}
\] 
are given by precomposition with the canonical functors $\check C(V_{12}) \to \check C(V_i)$. 
Here the groupoid $X^\mathbbm{2}$ is the arrow groupoid of $X$ and we are pulling back along the maps 
\[
	\inhom{\vC(V_{12})}{(X^\mathbbm{2})} \to \inhom{\vC(V_{12})}{X}
\]
which are given by postcomposition with the  functors $S,T\colon X^\mathbbm{2} \to X$ (on objects these are the usual source and target maps).

The space of arrows $\Map(S^1, X)_1$ is then 
\[
	\coprod_{V_1, V_2 \in C(S^1)}  \inhom{\vC(V_1)}{X} \times_{\inhom{\vC(V_{12})}{X}} \inhom{\vC(V_{12})}{(X^\mathbbm{2})} \times_{\inhom{\vC(V_{12})}{X}} \inhom{\vC(V_2)}{X}
\]
The source and target maps are projections on to the first and last factors. These are automatically smooth maps, and are both split by the unit map and hence are quotient maps; quotient maps in $\cD$ are subductions and hence the source and target maps are subductions.

Composition of transformations of anafunctors \cite[Proposition~12]{Bartels} (or \cite[Section 5]{Roberts_12} for a description closer to what is given here) is a little involved, but is essentially induced by the composition in $X^\mathbbm{2}$, which is smooth. This implies that composition in $\Map(S^1,X)$ is smooth, and hence that $\Map(S^1, X)$ is a diffeological groupoid.

\section{Presentation by a diffeological groupoid}
\label{sec:diffeological_presentation}

For $X$ a diffeological groupoid, to give a presentation over $\cD$ of the Hom-stack $\cHom(S^1,X)$ we will need to define a map from some diffeological space $A$, considered as a stack, to $\cHom(S^1,X)$.
Such a map is determined by a map of stacks $A \times S^1 \to \cX$.
Since these are all stacks arising from diffeological groupoids, this a map can be specified by constructing an anafunctor $A \times S^1 \gento X$ in the category of diffeological spaces as per Theorem~\ref{thm:anafunctors_are_stack_maps}.

Consider then the covers $V\to S^1$ used in the construction of the diffeological mapping groupoid, which are subductions since they admit local sections over open sets.
The product of subductions is again a subduction, so we can, for each $V \in C(S^1)$ define the anafunctor 
\[
	S^1 \times \inhom{\vC(V)}{X} \leftarrow \check C(V) \times \inhom{\vC(V)}{X} \xrightarrow{\ev} X
\]
where the right-pointing arrow is just the evaluation map for diffeological groupoids.
This gives us, via the preceeding argument, a map 
\[
	\inhom{\vC(V)}{X} \to \cHom(S^1,\cX)
\]
of stacks, and hence a map $q\colon \Map(S^1,X)_0 \to \cHom(S^1,\cX)$.

\begin{proposition}\label{prop:presentation_map_an_epi}

	For $X$ a diffeological groupoid, the map $q$ is an epimorphism of stacks.

\end{proposition}

\begin{proof}
	
	It is enough to show that for any $f\colon \mathbb{R}^n \to \cHom(S^1,\cX)$ there is an open cover of $\mathbb{R}^n$ and local sections of $q$ over it.
	The map $f$ is determined by a map $\mathbb{R}^n\times S^1 \to \cX$ of stacks, and hence an anafunctor ${F}\colon{\mathbb{R}^n\times S^1 \gento} X$ of diffeological groupoids. 
	Any subduction with codomain a manifold is refined by an open cover of the manifold, so we can replace the anafunctor by an isomorphic one of the form $\mathbb{R}^n\times S^1 \leftarrow \check C(U) \to X$, where $U = \coprod_i U_i \to \mathbb{R}^n\times S^1$ is an open cover.
	We can further repeat the argument from \cite[Proof of Theorem~4.2]{Noohi_10} 
	to construct an open cover $\coprod_j W_j \to \mathbb{R}^n$ 
	and for each $j$ an open cover $V^o_j\to S^1$ with closure 
	$V_j\to S^1$ an element of $C(S^1)$.
	Then the restrictions $W_j\times \check C(V_j) \to X$ of $F$ give maps of diffeological spaces $W_j \to \inhom{\vC(V_j)}{X}$, i.e.\ local sections of $q$ over the open cover $\{W_j\}$.
\end{proof}

To show that the groupoid $\Map(S^1,X)$ presents the Hom-stack, we need to show that the comma object of the map $q$ with itself is the arrow space $\Map(S^1,X)_1$ of our diffeological groupoid.

First, we do indeed have a 2-commuting square
\[
	\xymatrix{
		\Map(S^1,X)_1 \ar[rr]^s \ar[dd]_t && \Map(S^1,X)_0 \ar[dd]^q_{\ }="s"\\
		\\
		\Map(S^1,X)_0 \ar[rr]_q^{\ }="t" && \cHom(S^1,\cX)
		\ar@{=>}"s";"t"
	}
\]
which we can see by considering a component of the top left corner labelled by $V_1,V_2 \in C(S^1)$. 
The projection
$$
	X^{\check C(V_1)}\times_{X^{\check C(V_{12})}}(X^\mathbbm{2})^{\check C(V_{12})}\times_{X^{\check C(V_{12})}} X^{\check C(V_2)} \to (X^\mathbbm{2})^{\check C(V_{12})}  
$$
can be unwound to give a natural transformation between the maps $q\circ s$ and $q\circ t$.

To show the above diagram is indeed a comma square, we shall show that for any Euclidean space $\mathbb{R}^n$ the diagram of groupoids 
\[
	\centerline{
	\xymatrix{
		X^{\check C(V_1)}\times_{X^{\check C(V_{12})}}(X^\mathbbm{2})^{\check C(V_{12})}\times_{X^{\check C(V_{12})}} X^{\check C(V_2)}(\mathbb{R}^n)
		\ar[rr] \ar[dd] && X^{\check C(V_1)}(\mathbb{R}^n) \ar[dd]^q_{\ }="s"\\
		\\
		X^{\check C(V_2)}(\mathbb{R}^n) \ar[rr]_q^{\ }="t" && \cHom(S^1,\cX)(\mathbb{R}^n)
		\ar@{=>}"s";"t"
	}
	}
\]
is a comma square. It is immediate that all but the bottom right corner are sets, so we need to show the canonical map
\begin{multline*}
	c\colon X^{\check C(V_1)}\times_{X^{\check C(V_{12})}}(X^\mathbbm{2})^{\check C(V_{12})}\times_{X^{\check C(V_{12})}} X^{\check C(V_2)}(\mathbb{R}^n) \\
	\longrightarrow X^{\check C(V_1)} \downarrow_{\cHom(S^1,\cX)(\mathbb{R}^n)} X^{\check C(V_2)}(\mathbb{R}^n)
\end{multline*}
is a bijection. These sets are as follows: 
\begin{multline*}
	X^{\check C(V_1)}\times_{X^{\check C(V_{12})}}(X^\mathbbm{2})^{\check C(V_{12})}\times_{X^{\check C(V_{12})}} X^{\check C(V_2)}(\mathbb{R}^n) \\
	\simeq \left\{ (f,\alpha,g)\, \left| \;
	\raisebox{4.5ex}{
	\xymatrix{
		\mathbb{R}^n \times \check C(V_{12}) \ar[rr] \ar[d] 
			&& \mathbb{R}^n \times \check C(V_1) 
			\ar[d]^{f}_{\ }="s" \\
		\mathbb{R}^n \times \check C(V_2) \ar[rr]_{g}^(.65){\ }="t" && X 
		\ar@{=>}"s";"t"_{\alpha}
	}
	} \right.
	\right\}
\end{multline*}
and
\begin{multline*}
	X^{\check C(V_1)} \downarrow_{\cHom(S^1,\cX)(\mathbb{R}^n)} X^{\check C(V_2)}(\mathbb{R}^n) \\
	\simeq 
	\left\{ (\tilde f,\tilde \alpha,\tilde g)\, \left| \;
	\raisebox{4.5ex}{
	\xymatrix{
		\mathbb{R}^n \ar[rr]^{\tilde f} \ar[d]_{\tilde g} 
			&& X^{\check C(V_1)} 
			\ar[d]^{q}_(.45){\ }="s" \\
		X^{\check C(V_2)} \ar[rr]_{q}^{\ }="t" && \cHom(S^1,\cX)
		\ar@{=>}"s";"t"_{\tilde\alpha}
	}
	} \right.
	\right\}.
\end{multline*}
It is not difficult to see that $c$ must be injective; in particular $f$ and $g$ correspond to $\tilde f$ and $\tilde g$, respectively. 
Unravelling the description of $\tilde \alpha$ we can see it must arise from some $\alpha$ as in the first set, and so the map is bijective, and hence $X^{\check C(V_1)} \downarrow_{\cHom(S^1,\cX)(\mathbb{R}^n)} X^{\check C(V_2)}$ is representable, by the component of $\Map(S^1,X)_1$ labelled by $V_1,V_2$.

Since $q$ is an epimorphism, we then see that $q$ is representable (by Lemma~\ref{lemma:representable_maps_local_on_target}) and hence

\begin{theorem}

	For $X$ a diffeological groupoid, the Hom-stack $\cHom(S^1,X)$ is presented by the diffeological groupoid $\Map(S^1,X)$.

\end{theorem}

Note that there was nothing special about $S^1$ in this argument: we only required $S^1$ to be a manifold in order for the proof of Proposition~\ref{prop:presentation_map_an_epi} to work.
However for the next section the analysis is more delicate and so we have only treated the case of $S^1$.

\section{Presentation by a Fr\'echet--Lie groupoid}\label{sec:Frechet}

The diffeological groupoid $\Map(S^1, X)$ can also be considered in the case that $X$ is a Lie groupoid. In this section we will show that whenever $X$ is a Lie groupoid, the diffeological groupoid $\Map (S^1, X)$ defined in section \ref{section:diffeological_groupoid} is in fact a Fr\'echet--Lie groupoid (Theorem \ref{thm:mapping-groupoid-is-Frechet-Lie}) and that it also weakly presents the Hom-stack $\cHom(S^1,\cX)$ over the site of manifolds (Theorem \ref{thm:hom-stack-weakly-differentiable}).

From now on we will work with the coverage $C$ as in section \ref{section:diffeological_groupoid} but we will always use \emph{minimal} covers of $S^1$ (those such that triple intersections are empty), which are cofinal in $C(S^1)$. We denote the set of these minimal covers by $C(S^1)_{\text{min}}$ The object space is then 
\[
	\coprod_{V\in C(S^1)_\text{min}} \inhom{\vC(V)}{X},
\]
where again each component $\inhom{\vC(V)}{X}$ is the space of (smooth) functors $\check C(V) \to X$. 
This is naturally described as the iterated pullback 
\[
	X_0^{I_1} \times_{X_0^{J_1}} X_1^{J_1} \times_{X_0^{J_1}} X_0^{I_2}
	\times_{X_0^{J_2}}  X_1^{J_2} \times_{X_0^{J_2}}  \ldots \times_{X_0^{J_{n-1}}} X_0^{I_n}
\]
where $I_i$ are closed subintervals of $S^1$, $V = \coprod_{i=1}^n I_i$ and $J_i = I_i \cap I_{i+1}$, the maps
\[
	X_0^{I_{i}} \to X_0^{J_i} \leftarrow X_0^{I_{i+1}}
\]
are given by restriction, and the maps $X_1^{J_i} \to X_0^{J_i}$ are
induced by the source and target maps alternately. Here the $I_i$ and $J_i$ are intervals so that a functor $\check C(V) \to X$ consists of a series of paths $I_i \to X_0$ and
a series of paths $J_i \to X_1$ that ``patch together using source and
target''.

Recall that a pullback of a submersion in the category of Fr\'echet manifolds exists, and is again a submersion. Our strategy is to show that the maps above are all submersions, which will imply that the object space is a Fr\'echet manifold.

The following result of Stacey \cite{Stacey_13} guarantees that the maps $X_1^{J_i} \to X_0^{J_i}$ are submersions; see also \cite[Lemma 2.4]{Amiri-Schmeding}.

\begin{theorem}[Stacey]\label{Staceys_thm}
 Let $M \to N$ be a submersion of finite-dimensional manifolds and $K$ a compact
manifold. Then the induced map of Fr\'echet manifolds
$M^K \to N^K$ is a submersion.
\end{theorem}

For the maps $X_0^{I_{i}} \to X_0^{J_i} \leftarrow X_0^{I_{i+1}}$ we will need the following theorem, which may be derived from
the result in \cite{Seely_64} (see also \cite[\S 7]{Mitjagin_61}, which essentially proves Corollary~\ref{Corollary_to_Seely} directly).

\begin{theorem}[Seely]
The Fr\'echet space $(\RR^n)^{\RR_+}$ is a direct summand of
$(\RR^n)^\RR$, where we take the topology of uniform convergence of all
derivatives on compact subsets.
\end{theorem}

\begin{corollary}\label{Corollary_to_Seely}
The Fr\'echet space $(\RR^n)^{[0,1]}$ is a direct summand of
$(\RR^n)^{[-1,1]}$, hence the restriction map $(\RR^n)^{[-1,1]} \to (\RR^n)^{[0,1]}$
is a submersion of Fr\'echet spaces. The same is true with $[0,1]\subset [-1,1]$ replaced with any inclusion $J \subset I$ of compact intervals.
\end{corollary}

This allows us to prove

\begin{proposition}\label{restriction_is_subm}
Let $M$ be an $n$-dimensional manifold and $J \subset I$ two compact intervals.
Then the restriction map $M^I \to M^J$ is a submersion of Fr\'echet manifolds.
\end{proposition}

\begin{proof}  
Let $f\colon I \to M$ be a smooth function, and denote by $f_J\colon J \to M$ its restriction along the inclusion. To show that $M^I \to M^J$ is a submersion, we need to find charts around $f$ and $f_J$ such that the map is a submersion of Fr\'echet spaces on those charts. Recall \cite[\S I.4.1]{Hamilton_82} that a chart around $f$ is a neighbourhood of the zero section in $\Gamma(I,I\times_{f,M} TM)$, and similarly for $f_J$. Clearly $I\times_{f,M} TM \simeq I\times \mathbb{R}^n$, and given such an isomorphism we get an induced isomorphism $J\times_{f_J,M} TM \simeq J\times\mathbb{R}^n$ that is compatible with the restriction map. The induced map on spaces of sections,
\[
	(\mathbb{R}^n)^I = \Gamma(I,I\times\mathbb{R}^n) \to \Gamma(J,J\times\mathbb{R}^n) = (\mathbb{R}^n)^J,
\]
is just the obvious restriction map, and this map is locally the same, after unwinding the isomorphisms just given, to the restriction map. But Corollary \ref{Corollary_to_Seely} says that this map is a submersion, as we needed.
\end{proof}

Proposition~\ref{restriction_is_subm} implies that the maps $X_0^{I_{i}} \to X_0^{J_i} \leftarrow X_0^{I_{i+1}}$ are submersions and hence we have

\begin{proposition}
For $X$ a Lie groupoid, the object space $\Map(S^1, X)_0$ is a Fr\'echet manifold.
\end{proposition}

To see that the set of arrows has a manifold structure as well, recall that this set is given by 
\[
	\coprod_{V_1, V_2 \in C(S^1)_\text{min}}  \inhom{\vC(V_1)}{X} \times_{\inhom{\vC(V_{12})}{X}} \inhom{\vC(V_{12})}{(X^\mathbbm{2})} \times_{\inhom{\vC(V_{12})}{X}} \inhom{\vC(V_2)}{X}
\]
where the chosen refinement $V_{12}$ is also a minimal cover of $S^1$. To use the same reasoning as above we need to know that the maps 
\[
	\inhom{\vC(V_{12})}{(X^\mathbbm{2})} \to \inhom{\vC(V_{12})}{X},
\]
induced by $S,T\colon X^\mathbbm{2} \to X$, and 
\[
	\inhom{\vC(V_i)}{X} \to \inhom{\vC(V_{12})}{X}
\]
($i = 1, 2$) are submersions.

Now for $M \to N$ a map of finite-dimensional manifolds, and $C \to D$ a map of compact manifolds with boundary, the two induced maps
\[
	M^C \to N^C,\quad \text{and}\quad M^D \to M^D
\]
have a rather nice property in that on certain canonical charts they are actually \emph{linear} maps (recall that these maps above look locally like maps between spaces of sections induced by vector bundle maps). 
More generally one can consider larger diagrams, all of whose maps have this local linearity, and further the charts exhibiting this local behaviour can all be chosen compatibly.
Such a diagram will be called be called \emph{locally linear}.

An example of such a diagram is one where all the objects are mapping spaces as above, and all arrows are induced by pre- or post-composition as above.
A much simpler and familiar example would be in the finite-dimensional setting, where the exponential map is a local diffeomorphism. 
The induced diagram on tangent spaces, for any compatible system of basepoints, is then a diagram of vector spaces.

We have the following Lemma:

\begin{lemma}\label{lemma:map_of_wide_pullbacks}
Let
\[
	\xymatrix{
		A_1 \ar[d] \ar[r] & A_2 \ar[d] & \ar[l] A_3 \ar[r] \ar[d] & \cdots & \ar[l] A_n \ar[d]\\
		B_1 \ar[r] & B_2 & \ar[l] B_3 \ar[r]  & \cdots & \ar[l] B_n 
	}
\]
be a diagram of submersions that is locally linear. 
Then the natural map 
\[
	\lim A_i \to \lim B_i,
\]
where the limits are iterated fibre products, is also a submersion.
\end{lemma}

\begin{proof}

	The local linearity of the diagram means that we can find a diagram of the same shape in the category of Fr\'echet \emph{spaces} and linear maps, and in fact \emph{split} linear maps, since all of the maps are submersions, hence locally split.
	Then the proof that the induced map is a split submersion of Fr\'echet spaces proceeds exactly as one would in the finite-dimensional case. 
	One can induct on the length of the zig-zags and so reduce to the case of a diagram
	\[
		\xymatrix{
			A_1 \ar[d] \ar[r] & A_2 \ar[d] & \ar[l] A_3 \ar[d] \\
			B_1 \ar[r] & B_2 & \ar[l] B_3 
		}
	\]
	in the category of Fr\'echet spaces and linear maps and then show that one can find a section of the linear map $A_1\times_{A_2} A_3 \to B_1\times_{B_2} B_3$.
\end{proof}

Let $X \to Y$ be a functor between Lie groupoids such that the object and arrow components are submersions.
We call such a functor \emph{submersive}. 
We have the following result.

\begin{lemma}\label{lemma:submersive_functor_gives_submersion}\hfill
\begin{enumerate}
\item
Let $X \to Y$ be a submersive functor between Lie groupoids. 
Then the induced map
\[
	\inhom{\vC(V)}{X} \to \inhom{\vC(V)}{Y}
\]
is a submersion.
\item
Let $X$ be a Lie groupoid and $V_1 \to V_2$ be a refinement of minimal covers. 
Then the induced map
\[
	\inhom{\vC(V_2)}{X} \to \inhom{\vC(V_1)}{X}
\]
is a submersion.
\end{enumerate}
\end{lemma}

\begin{proof}
The first part follows from Theorem~\ref{Staceys_thm} and Lemma~\ref{lemma:map_of_wide_pullbacks} and the second follows from Proposition~\ref{restriction_is_subm} and Lemma~\ref{lemma:map_of_wide_pullbacks}.
\end{proof}

Lemma~\ref{lemma:submersive_functor_gives_submersion} implies that the maps above are submersions and so we have

\begin{proposition}
For $X$ a Lie groupoid, the arrow space $\Map(S^1, X)_1$ is a Fr\'echet manifold.
\end{proposition}

Now happily, the source and target map for our Fr\'echet--Lie groupoid
are given, on each component of the arrow Fr\'echet manifold, by
the two projections
\[
		\inhom{\vC(V_1)}{X} \times_{\inhom{\vC(V_{12})}{X}} \inhom{\vC(V_{12})}{(X^\mathbbm{2})} \times_{\inhom{\vC(V_{12})}{X}} \inhom{\vC(V_2)}{X} \to \inhom{\vC(V_i)}{X}
\]
where $i=1,2$, which are submersions. Therefore

\begin{theorem}\label{thm:mapping-groupoid-is-Frechet-Lie}
For $X$ a Lie groupoid, $\Map(S^1, X)$ is a Fr\'echet--Lie groupoid.
\end{theorem}

Observe that $\LL X$ is built by taking disjoint unions of pullbacks of smooth path spaces, and smooth path spaces are metrisable and smoothly paracompact (as they are nuclear Fr\'echet spaces).
By a combination of Lemma~27.9 and the comments in \S27.11 of \cite{KM97}, the pullback $M_1 \times_N M_2$, where $M_1,M_2$ are metrisable smoothly paracompect and where at least one of $M_i \to N$ is a submersion, is smoothly paracompact.
Thus by induction the iterated pullback that defines $\inhom{\cech{V}}{X}$ is a smoothly paracompact manifold, and so the object and arrow manifolds of $\LL X$ are smoothly paracompact.
This means that every open cover admits subordinate \emph{smooth} partitions of unity, and so any geometric constructions with smooth objects (differential forms and so on) can be built locally.

In fact the spaces $\LL X_n$ of sequences of $n$ composable arrows are also paracompact, so that $\LL X$ is a paracompact groupoid in the terminology of Gepner--Henriques.
As a result we know that the fat geometric realisation $||\LL X||$ of the nerve of $\LL X$ is a paracompact space \cite[Lemma~2.25]{Gepner-Henriques_07}.

The following Proposition means that the endo-2-functor on stacks on $\cM$ lifts to a 2-functor on \emph{presentations} of stacks.
It is thus a kind of rigidification of the loop stack functor.

\begin{proposition}

	The assignment $X\mapsto \LL X$ extends to a 2-functor 
	\[\LL\colon \Gpd(\cM)\to \Gpd(\cF).\]

\end{proposition}

\begin{proof}

Given a functor $f\colon X\to Y$ between Lie groupoids, we clearly get a functor $\LL f\colon \LL X \to \LL Y$ between Fr\'echet--Lie groupoids, by composing everything in sight with $f$.
Moreover, given a second functor $k\colon Y\to Z$, we clearly have $\LL(kf) = \LL k\, \LL f$.

Assume now that we have a natural transformation $\alpha\colon f\Rightarrow g\colon X\to Y$, or in other words a functor $X \to Y^\mathbbm{2}$. 
We need to show that this induces a natural transformation $\LL f \Rightarrow \LL g$, which is determined by the data of a smooth map
\[
	\Map(S^1,X)_0 \to \Map(S^1,Y)_1.
\]
We first need to describe this map on the level of underlying sets.
Let $S^1 \leftarrow \vC(V) \xrightarrow{h} X$ be an anafunctor. 
The value of the natural transformation $\LL\alpha \colon \LL f \Rightarrow \LL g$ at $h$ is a transformation of anafunctors
\[
	 \LL\alpha(h) \colon (\vC(V) \xrightarrow{fh} Y) \Rightarrow (\vC(V) \xrightarrow{gh} Y)
\]
and so lives in the component
\[
	\inhom{\vC(V)}{Y} \times_{\inhom{\vC(V)}{Y}} \inhom{\vC(V)}{(Y^\mathbbm{2})} \times_{\inhom{\vC(V)}{Y}} \inhom{\vC(V)}{Y}
	\simeq \inhom{\vC(V)}{(Y^\mathbbm{2})}
\]
Moreover, the transformation $\LL\alpha(h)$ is simply the left whiskering of $\alpha$ by the functor $h$. 
Thus $\LL\alpha$ is given (on one component) by the map
\[
	\inhom{\vC(V)}{X} \to \inhom{\vC(V)}{(Y^\mathbbm{2})},
\]
induced by composition with the given $X \to Y^\mathbbm{2}$, hence the component map of the natural transformation $\LL f\Rightarrow \LL g$ is smooth.

Now it remains to show firstly that $\LL\alpha$ is natural, and secondly that this is functorial for both compositions of 2-cells.
Naturality follows from the proof that anafunctors are 1-cells in a bicategory, and that functors are 1-cells in the locally full sub-bicategory $\Gpd(\cM)$.
Functoriality follows from the fact whiskering is a functorial process.
\end{proof}

\begin{remark}

	The 2-functor $\LL\colon \Gpd(\cM)\to \Gpd(\cF)$ preserves products up to weak equivalence. 
	This follows formally using the equivalence between differentiable stacks and Lie groupoids and anafunctors, and the fact that the product of differentiable stacks is presented by the product of Lie groupods.
	However we actually have a slightly more rigid result, with the coherence functor (in one direction) being the canonical inclusion
	\[
		\LL (X\times Y) \into \LL X \times \LL Y,
	\]
	rather than some comparison anafunctor.
	This has a quasi-inverse \emph{functor} that takes a pair of objects, in summands indexed by the covers $V_1$ and $V_2$ respectively, to the isomorphic pair indexed by the same cover $V_{12}$, the chosen common refinement of $V_1$ and $V_2$.

\end{remark}

Now the construction of the map $q$ from section \ref{sec:diffeological_presentation} is identical, we need to additionally show that it is a submersion.
There is a small subtlety here, in that we haven't been able to show directly that $q$ is a representable map of stacks, rather we will rely on (a submersion variant of) the weaker notion of presentation from \cite[\S 6.2.0.1]{Pronk_96}, which only requires that the comma object of $q$ with itself gives a submersion between manifolds.
Since we know the comma object $q\downarrow q$ is already a manifold, namely the arrow space $\Map(S^1,X)_1$, and the projections are the source and target maps, which are submersions, then we have our first main result.

\begin{theorem}\label{thm:hom-stack-weakly-differentiable}

	For $X$ a Lie groupoid, the Hom-stack $\cHom(S^1,\cX)$ is weakly presented by the Fr\'echet--Lie groupoid $\Map(S^1,X)$.

\end{theorem}

As the stack $\cHom(M,\cX)$ is presented by a paracompact groupoid it is well-behaved homotopically.
Proposition~8.5 in \cite{Noohi_12} ensures that since $\LL X$ has object and arrow manifolds metrisable, $\cHom(M,\cX)$ has a \emph{hoparacompact} underlying topological stack.
Then the classifying space of $\cHom(M,\cX)$ (as defined in \cite{Noohi_12}) is well-defined up to homotopy equivalence, rather than weak homotopy equivalence.

We note that with minor modifications, one can repeat the above analysis for the case of the mapping stack $\cHom([0,1],\cX)$, but we leave that to the interested reader.

\section{Recap on differentiable gerbes}
\label{sec:gerbes}

\begin{definition}
A (Fr\'echet-)Lie groupoid $X\to M$ is a \emph{gerbe} if $\pi\colon X_0\to M$ and $(s,t)\colon X_1\to X_0^{[2]}$ are surjective submersions. 
The stack on $\cM$ that such a groupoid (weakly) presents will be called a (Fr\'echet-)differentiable gerbe.
\end{definition}

Equivalently, we can require that $X \to M$ and $X \to \vC(X_0)$ are submersive functors that are surjective on objects and arrows. We rephrase these properties in terms of functors rather than component maps because later we wish to prove stability of these properties under forming mapping groupoids.

\begin{remark}
In this section the results also apply to general Fr\'echet--Lie groupoids, even though we have only stated them for Lie groupoids for brevity.
\end{remark}

Because $(s,t)$ is a submersion the pullback $\Lambda X := \Delta^* X_1 \to X_0$, for $\Delta\colon X_0 \to X_0^{[2]}$ the diagonal, is a bundle of Lie groups. We shall call $\Lambda X$ the \emph{inertia bundle}. If the fibre $\Lambda X_x \simeq G$ for every $x\in X_0$ then this gives a \emph{$G$-gerbe} in the sense of \cite{LGSX}, that is, an extension of groupoids
\[
	\Lambda X \to X \to \vC(X_0).
\]
However, we wish to use a mental picture as close to bundle gerbes \cite{Mur} as possible, so offer the following diagram 
encoding a gerbe $X\to M$:
\[
	\xymatrix{
		X_1 \ar[d] & \Lambda X \ar[d]\\
		X_0^{[2]} \ar@<1ex>[r] \ar@<-1ex>[r] & X_0 \ar[d] \\
		& M
	}
\]
We have left and right actions of $\Lambda X$ on $X_1$, or rather a left action of $\Lambda X_L :=\pr_1^* \Lambda X$ and a right action of $\Lambda X_R :=\pr_2^* \Lambda X$ on $X_1$, preserving the fibres of $(s,t)$, by composition in the groupoid $X$. 
This makes $X_1 \to X_0^{[2]}$ a principal $\Lambda X_L$-$\Lambda X_R$-bibundle. 
Notice that $X_1$ is locally isomorphic to $\Lambda X_L$ and to $\Lambda X_R$ (as spaces over $X_0^{[2]}$) using local sections of $(s,t)$.

There is also an action of $X$ as a groupoid on the family $\Lambda X \to X_0$, covering the action of $X$ on $X_0$. 
This is by conjugation in the groupoid: if $f\colon x\to y \in X_1$ and $\alpha \in \Lambda X_x$, then $f^{-1}\alpha f \in \Lambda X_y$, where we are using the diagrammatic (or algebraic) composition order. 
This defines a smooth map
\[
	\Lambda X \times_{X_0,s} X_1  \to \Lambda X
\]
over $X_0$, using the target map composed with the second projection on the domain. We also want to think of this in the equivalent form of
\[
	 \Lambda X_L \times_{X_0^{[2]}} X_1   \to \Lambda X_R,
\]
a map over $X_0^{[2]}$. 
This action defines an action groupoid $\Lambda X /\!/ X$ with objects $\Lambda X$ and morphisms $\Lambda X_L \times_{X_0^{[2]}} X_1$. 
This groupoid will become important for calculations in the next section. 
We will denote an object and an arrow of $\Lambda X /\!/ X$ by
\[
	\xymatrix{x \ar@(ul,dl)_{\alpha} } \qquad \text{and} \qquad \xymatrix{
		x \ar@(ul,dl)_{\alpha}   \ar[r]^f &  y
	}\;,
\]
respectively. 
The action of $X$ on $\Lambda X$, that is, the target map of $\Lambda X /\!/ X$, is 
\begin{equation}\label{eq:action_of_X_on_LambdaX}
	\xymatrix{
			x \ar@(ul,dl)_{\alpha}   \ar[r]^f &  y 
	}\qquad \longmapsto \quad
	\xymatrix{
		y \ar@(ul,dl)_{f^{-1}\alpha f}
	}
\end{equation}

For the purposes of being confident that various pullbacks exist in what follows, we record some trivial consequences of the conditions on the definition of a gerbe. 
Note that the surjectivity requirements are superfluous at this point, but will become important later.

\begin{lemma}\label{lemma:some_standard_submersive_functors}
For a Lie groupoid $X$ with submersive functors $X\to M$ and $X\to \vC(X_0)$, the following functors are also submersive:
\begin{enumerate}
\item $(S,T)\colon X^{\mathbbm{2}} \to X\times_MX$
\item $\pr_i\colon X\times_MX \to X$ for $i=1,2$
\item $S,T\colon X^{\mathbbm{2}} \to X$
\end{enumerate}
\end{lemma}

While there may be some utility in maintaining extra generality at this point, our results will ultimately be applied in the case that $\Lambda X$ is a bundle of \emph{abelian} Lie groups. 
Thus from now on we make this assumption.
Note however that an \emph{abelian gerbe} in the sense of \cite[Definition 2.9]{Breen} is more restrictive than simply demanding $\Lambda X_x$ is abelian for every $x\in X_0$. 
We will get to this type of gerbe soon (see Definition \ref{def:abelian_gerbe} below)

We are also interested primarily in the case that $\Lambda X\to X_0$ descends to $M$. 
This means that there is an isomorphism
\[
	\phi\colon \Lambda X_L \xrightarrow{\sim} \Lambda X_R 
\]
over $X_0^{[2]}$ which satisfies the cocycle condition over $X_0^{[3]}$. We will refer to $\phi$ as the \emph{descent isomorphism} for $\Lambda X$.
We can denote this isomorphism by
\[
	\raisebox{4ex}{
		\xymatrix{
			x \ar@(ul,dl)_{\alpha}  \\ y
		}
	}\qquad \stackrel{\simeq}{\longmapsto} \quad
	\raisebox{4ex}{
		\xymatrix{
			 x  \\ y \ar@(ul,dl)_{\phi(\alpha)}
		}
	}
\]
where $(x,y) \in X_0\times_M X_0$.
There is then a bundle of groups $\mathcal{A}\to M$ such that $\pi^* \mathcal{A} \simeq \Lambda X$. 
Another way to phrase this is that there is an action $\Lambda X \times_{X_0,\pr_1} X_0^{[2]} = \Lambda X_L \to \Lambda X$ of the groupoid $\vC(X_0)$ on $\Lambda X$, and hence we have an action groupoid $\Lambda X/\!/ \vC(X_0)$ with arrows $\Lambda X_L$. 
This has a projection map to $\vC(X_0)$ making $\Lambda X/\!/ \vC(X_0) \to \vC(X_0)$ a bundle of groups object in the category of Lie groupoids.

\begin{lemma}\label{lemma:A-gerbe}
If $\Lambda X$ descends to $\mathcal{A}$ on $M$, the following square is a pullback of Lie groupoids
\[
	\xymatrix{
		\Lambda X/\!/ \vC(X_0) \ar[r] \ar[d] & \mathcal{A} \ar[d] \\
		\vC(X_0) \ar[r] & M \,.
	}
\]
\end{lemma}

\begin{proof}
We can verify this by looking at the level of objects and arrows, individually. 
The object manifold of $\Lambda X/\!/ \vC(X_0)$ is $\Lambda X$, and by assumption this is isomorphic to $X_0\times_M \mathcal{A}$, as needed.
The square at the level of arrow manifolds is 
\[
	\xymatrix{
		\Lambda X \times_{X_0,\pr_1} (X_0\times_M X_0) \ar[r] \ar[d] & \mathcal{A} \ar[d]\\
		X_0\times_M X_0 \ar[r] & M
	}
\]
and so we need to show that the induced map 
\begin{equation}\label{eq:bundle of groups descent comparison on arrows}
	\Lambda X \times_{X_0,\pr_1} (X_0\times_M X_0) \to (X_0\times_M X_0)\times_M \mathcal{A}
\end{equation}
is an isomorphism. 
But $\Lambda X \simeq  X_0\times_M \mathcal{A}$, so (\ref{eq:bundle of groups descent comparison on arrows}) is just the canonical isomorphism rearranging the factors of a iterated pullback.
\end{proof}

We can hence talk about \emph{$\mathcal{A}$-gerbes on $M$} for a fixed bundle of abelian groups $\mathcal{A}\to M$, and we will restrict attention to this case from now on. The bundle $\mathcal{A}$ will be referred to as the \emph{structure group bundle}.

To go further and talk about \emph{abelian} $\mathcal{A}$-gerbes we need to say what it means for the left and right actions of $\Lambda X_L$ and $\Lambda X_R$ on $X_1$ to agree. In the special case that $\Lambda X = X_0\times A$, then $\Lambda X_L = X_0^{[2]} \times A = \Lambda X_R$, and we could ask that the $A$-$A$-bibundle $X_1$ is in fact just an $A$-bundle, with the right action equal to the left action.

In the case that $\Lambda X = \pi^*\mathcal{A}$ is non-trivial, the best we can do is identify $\Lambda X_L$ with $\Lambda X_R$ via the given descent isomorphism, and ask that \emph{relative to this identification}, the left and right actions agree. There are two ways to look at this agreement, from the point of view of the actions of $\Lambda X_L$, $\Lambda X_R$ on $X_1$, or the action of $X$ on $\Lambda X$. 
However, we want to also introduce a third way, that uses a more global, groupoid-based approach to be used in the next section.

Recall that $X^\mathbbm{2}$ is the arrows of a groupoid object in Lie groupoids---that is, a double groupoid. 
There is a groupoid action \emph{in the category of Lie groupoids}
\[
	\Lambda X/\!/ X  \times_{X,S} X^\mathbbm{2} \to \Lambda X/\!/ X 
\]
with the object component of this functor given by equation (\ref{eq:action_of_X_on_LambdaX}).
The arrow component is given by
\[
	\left(
	\xymatrix{
		x \ar@(ul,dl)_{\alpha}  \ar[r]^g & y }
		\;,\;
	\raisebox{4ex}{\xymatrix{
		x \ar[r]^g \ar[d]_f & \ar[d] y \\
		 z \ar[r]_h & w
	} 
	}\right)
	\qquad \longmapsto \quad 
	\xymatrix{
		z \ar@(ul,dl)_{f^{-1}\alpha f}   \ar[r]^h & w
	}
\]
We remind the reader that here notation for the conjugation action is using the diagrammatic order for composition.

While this seems to iterate our data to another level of complexity, this allows us to consider stability of structures under the functor 
\[
	(-)^{\vC(V)}\colon \Gpd(\cM) \to \cF. 
\]	
In particular, since $(-)^{\vC(V)}$ preserves products and even pullbacks of submersive functors, for a bundle of groups $\mathcal{G} \to X$ in $\Gpd(\cM)$ (considered as a 1-category), $\mathcal{G}^{\vC(V)} \to X^{\vC(V)}$ is a bundle of Fr\'echet--Lie groups. This will allow a calculation of the structure group bundle of the (putative) gerbe $\LL X$, once we prove that it is in fact a gerbe.

\begin{lemma}\label{lemma:LR_actions_agree}

	For $X$ an $\mathcal{A}$-gerbe, the following are equivalent:
	\begin{enumerate}
		\item The diagram
		\[
			\xymatrix{
				\Lambda X_L  \times_{X_0^{[2]}} X_1 
				\ar[r]^-{\simeq} \ar[dd]_{\phi\times \id_{X_1}} & 
				\Lambda X \times_{X_0,s} X_1 \ar[dr] \\
				&& X_1\\
				\Lambda X_R \times_{X_0^{[2]}} X_1 
				\ar[r]^-{\simeq} &
				X_1 \times_{t,X_0} \Lambda X \ar[ur]
			}
		\]
		sitting over $X_0^{[2]}$ commutes (``the right and left actions of $\Lambda X$ on $X_1$ agree'');

		\item The conjugation action of $X$ on $\Lambda X$ factors through the action of $\vC(X_0)$ on $\Lambda X$, via the projection $X\to \vC(X_0)$;
		
		\item The action of $X^\mathbbm{2}$ on $\Lambda X/\!/ X$ factors through an action 
		\begin{equation}\label{eq:double_Cech_action}
			\Lambda X/\!/ X  \times_{X,\pr_1} (X\times_M X) \to \Lambda X/\!/ X, 
		\end{equation}
		of the double groupoid $X\times_M X \rightrightarrows X$ on $\Lambda X/\!/ X$, in the category of Lie groupoids, whose object component is the descent isomorphism $\phi$ for $\Lambda X$, via the functor $(S,T)\colon X^\mathbbm{2} \to X\times_M X$. 
	\end{enumerate}
\end{lemma}
\begin{proof}
We will first prove that 1.\ and 2.\ are equivalent. The implication 3.$\Rightarrow$2.\ is immediate because 2.\ is merely the object component of 3. We will then show how 3.\ follows from 2. 

The diagram in 1.\ commuting means that for all $(\alpha,f)\in \Lambda X_L\times_{X_0^{[2]}} X_1$, $\alpha f = f\phi(\alpha)$.
In other words, that $\phi(\alpha) = f^{-1}\alpha f$, but this is precisely what it means for the action of $X$ on $\Lambda X$ to factor through the action of $\vC(X_0)$ on $\Lambda X$, and so 1.$\Leftrightarrow$2.

To prove that 2.\ implies 3., we need first to describe an action as in (\ref{eq:double_Cech_action}) with object component $\Lambda X_L = \Lambda X \times_{X_0,\pr_2}X_0^{[2]} \xrightarrow{\phi} \Lambda X_R \xrightarrow{\pr} \Lambda X$.
If 2.\ holds then 
\begin{equation}\label{eq:descent_iso_vs_conjugation}
	\raisebox{4ex}{
		\xymatrix{
			x \ar@(ul,dl)_{\alpha}  \\ y
		}
	}\qquad \stackrel{\simeq}{\longmapsto} \quad
	\raisebox{4ex}{
		\xymatrix{
			 x  \\ y \ar@(ul,dl)_{\phi(\alpha)}
		}
	} \qquad = \quad
	\raisebox{4ex}{
		\xymatrix{
			 x  \\ y \ar@(ul,dl)_{f^{-1}\alpha f}
		}
	}
\end{equation}
for any $f\colon x\to y \in X_1$.
Hence we can define the arrow component of (\ref{eq:double_Cech_action}) by
\[
	\left(
	\xymatrix{
		x \ar@(ul,dl)_{\alpha}  \ar[r]^g & y }
		\;,\;
	\raisebox{4ex}{\xymatrix{
		x \ar[r]^g  & y \\
		 z \ar[r]_h & w
	} 
	}\right)
	\qquad \longmapsto \quad 
	\xymatrix{
		z \ar@(ul,dl)_{\phi(\alpha)}   \ar[r]^h & w
	}
\]
which is indeed a functor by virtue of (\ref{eq:descent_iso_vs_conjugation}), and the fact it is an action follows from the cocycle identity for $\phi$.
The action of $X^\mathbbm{2}$ on $\Lambda X/\!/ X$ factors through this action by construction.

The arrow component of the action in 3.\ is in fact determined uniquely, rather than merely being `an' action, since 3.\ implies 2.\ and then one can construct the required arrow component of the action functor.
\end{proof}

Note that in particular that if the conditions of the lemma are satsified there is a functor $\Lambda X/\!/ X \to \Lambda X /\!/ \vC(X_0)$ sitting over $X \to \vC(X_0)$.

\begin{definition}\label{def:abelian_gerbe}
	We call an $\mathcal{A}$-gerbe $X\to M$ \emph{abelian} if the equivalent conditions of Lemma~\ref{lemma:LR_actions_agree} hold.
\end{definition}

\begin{lemma}\label{lemma:inertia_with_adj_action_pullback_from_base}
	For an abelian $\mathcal{A}$-gerbe $X\to M$, the left and hence all squares below are pullbacks of Lie groupoids
	\[
		\xymatrix{
			\Lambda X/\!/ X \ar[r] \ar[d] & \Lambda X /\!/ \vC(X_0) \ar[r] \ar[d] & \mathcal{A} \ar[d]\\
			X \ar[r] & \vC(X_0) \ar[r]& M\, .
		}
	\]
	In particular, $\Lambda X/\!/ X \to X$ is a bundle of groups object in the category of Lie groupoids.
\end{lemma}

\begin{proof}
Since $X$ is an abelian $\mathcal{A}$-gerbe, and hence an $\mathcal{A}$-gerbe, the right square is a pullback by Lemma~\ref{lemma:A-gerbe}.
By the pullback pasting lemma, the left square is a pullback if and only if the outer rectangle is a pullback; we shall prove the former. 
The object components of the top and bottom horizontal functors in the left square are identity maps $\id_{\Lambda X}$ and $\id_{X_0}$ respectively, and the left and right vertical maps are both $\Lambda X \to X_0$, hence on objects the left square is a pullback. 
The morphism components of the left square give the square
\[
	\xymatrix{
			\Lambda X_L \times_{X_0^{[2]}} X_1  \ar[rr]^{\id\times(s,t)} \ar[d]_{\pr_2} && \Lambda X_L \times_{X_0^{[2]}} X_0^{[2]} \ar[d]^{\pr_2} \\
			X_1 \ar[rr]_{(s,t)} && X_0^{[2]}\,,
		}
\]
which is manifestly a pullback.
\end{proof}

\begin{example}

	Let $A$ be an abelian Lie group. An $A$-bundle gerbe on $M$ in the sense of \cite{Mur} is an abelian $M\times A$ gerbe $X\to M$. Most often one just considers the case\footnote{%
		There is also a version of bundle gerbes where $X_1 \to X_0^{[2]}$ is a \emph{line bundle}, rather than a principal bundle. This is captured in our framework if we allow for Lie groupoids that are enriched over a monoidal category of smooth objects, in this case the category $\Lines_\mathbb{C}$ of complex lines with the usual tensor product. Asking that $X$ is enriched over $\Lines_\mathbb{C}$ in this internal setting is nothing other than asking that $X_1 \to X_0\times X_0$ is a line bundle over its image.
	} 
	of $A=U(1)$ or $\mathbb{C}^\times$.
	Note that the local triviality of the $A$-bundle $X_1 \to X_0\times_M X_0$ follows from the rest of the definition, as it is a surjective submersion, hence has local sections, and has an action by $A$ that is free and transitive on fibres. 

\end{example}

\begin{example}
	
	In \cite{HMSV_13} the second-named author and collaborators considered `bundle gerbes with non-constant structure group bundle'. This is a case intermediate between gerbes as defined here and ordinary bundle gerbes as in \cite{Mur}, requiring that $\mathcal{A}$ is a locally trivial bundle of groups, and $X_1 \to X_0^{[2]}$ is locally trivial in a way compatible with the induced local trivialisations of $\pi^* \mathcal{A}$.
	The main nontrivial example of \cite{HMSV_13} is however infinite-dimensional, meaning the results of the present paper can only be applied if we consider it as a diffeological groupoid.

\end{example}

In fact, assuming $\mathcal{A}$ is a locally trivial bundle of groups has consequences for the structure of abelian $\mathcal{A}$-gerbes.

\begin{lemma}\label{lemma:locally_trivial_A_gerbes}
	Let $\mathcal{A}\to M$ be a locally trivial bundle of abelian groups.
	Then for any abelian $\mathcal{A}$-gerbe $X\to M$, the map $(s,t)\colon X_1 \to X_0\times_M X_0$ is a locally trivial bundle.
\end{lemma}
\begin{proof}
The locally trivial bundles of groups $\Lambda X_L$ and $\Lambda X_R$---pullbacks of $\mathcal{A}$---act principally on $X_1$ (that is: freely, and transitively on the fibres of $(s,t)$), and $(s,t)$ admits local sections as it is a submersion.
From these local sections and local trivialisations of, say $\Lambda X_L$, we can construct local trivialisations of $X_1 \to X_0\times_M X_0$
\end{proof}

We end with a final technical lemma used in the next section, but of independent interest. 

\begin{lemma}\label{lemma:inertia_bundle_pullback_of_arrow_groupoid}
The bundle of groups $\Lambda X/\!/ X \to X$ (internal to Lie groupoids) is the pullback of $(S,T)\colon X^\mathbbm{2} \to X\times X$ along the diagonal $\Delta\colon X \to X\times X$. 
\end{lemma}

\begin{proof}
On the level of objects this says that $\Lambda X$ is the pullback of $X_1$ along $X_0 \to X_0 \times X_0$, which is true by definition.
The arrow manifold of $X^\mathbbm{2}$ can be described as the pullback of $(s,t)\colon X_1 \to X_0^2$ along $(s,s)\colon X_1^2 \to X_0^2$.
In this description, the (arrow component of the) source functor $S$ projects on the first factor of $X_1^2$, and the (arrow component of the) target functor $T$ projects on the other factor.
Thus the pullback of $X_1 \times_{X_0^2} X_1^2$ along the diagonal $X_1 \to X_1\times X_1$ forces the last two components to be equal, and hence that the middle factor must be $\Lambda X$, and the pullback is $X_1\times_{s,X_0} \Lambda X$ which is the arrow manifold of $\Lambda X /\!/X$.
\end{proof}

\section{The loop stack of a gerbe}
\label{sec:looped_gerbes}

This section shows that given a differentiable gerbe $\cX$ on a manifold $M$ presented by a Lie groupoid $X$ satisfying (a) a connectedness property for its automorphism groups $X(x,x)$ and (b) a weak form of local triviality of $X_1 \to X_0 \times_M X_0$; then the loop stack is again a (Fr\'echet) differentiable gerbe.
An example of such a groupoid is a bundle gerbe (see below), in which case $(s,t)$ is the projection map for a principal bundle.

In the following, denote $\Map(S^1, X)$ by $\LL X$. 
We will also denote $(\LL X)_i$, i.e.~the object and arrow manifolds, simply by $\LL X_i$.

\begin{proposition}
	
	Let $X$ be a Lie groupoid with a submersive functor $X\to \disc(M)$ such that the resulting map $X_1 \to X_0\times_M X_0$ is a submersion. 
	Then $\LL X\to \disc(LM)$ is submersive and $(s,t)^{\LL X} \colon \LL X_1 \to \LL X_0\times_{LM} \LL X_0$ is a submersion.

\end{proposition}

Note that we do not need to assume that $X$ presents a gerbe on $M$, so that the result will be applicable to more general bundles of groupoids, in particular those whose fibres are not necessarily transitive.

\begin{proof}

	Firstly, as $X\to \disc(M)$ is submersive we have the composite map $\inhom{\cech{V}}{X} \to \inhom{\cech{V}}{\disc(M)} \to LM$ a submersion (Lemma~\ref{lemma:submersive_functor_gives_submersion} parts 1 and 3), and so applying Lemma~\ref{lemma:map_of_wide_pullbacks} we get that $\LL X_0 \to LM$ is a submersion.
	Thus we know $\LL X_0\times_{LM}\LL X_0$ is a Fr\'echet manifold.

	The crux of the proof to show $(s,t)^{\LL X}$ is a submersion is in finding an isomorph\footnote{An \emph{isomorph} of a map $f\colon A \to B$ is a map $g\colon A' \to B'$ such that there are isomorphisms $A\simeq A'$ and $B\simeq B'$ making the resulting square commute. It is obvious that the isomorph of a submersion is a submersion.} of the map $(s,t)^{\LL X}$ in such a way that Lemmata \ref{lemma:map_of_wide_pullbacks} and \ref{lemma:submersive_functor_gives_submersion} can be applied.

	Firstly notice that we can work with $(s,t)^{\LL X}$ over each component of its domain and codomain, which are indexed by pairs $V_1,V_2$ of covers of $S^1$.
	This is because the disjoint union of submersions is again a submersion.
	Hence we are only dealing with the map
	\begin{equation}\label{eq:looped_source_target}
		\inhom{\cech{V_1}}{X} \times_{ \inhom{\cech{V_{12}}}{X}} 
		\inhom{\cech{V_{12}}}{(X^\mathbbm{2})}
		\times_{ \inhom{\cech{V_{12}}}{X}} \inhom{\cech{V_2}}{X} 
		\longrightarrow
		\inhom{\cech{V_1}}{X} \times_{LM} \inhom{\cech{V_2}}{X} 
	\end{equation}
	which is projection on the first and third factors of the domain.
	There are isomorphisms
	\begin{align*}
		\inhom{\cech{V_1}}{X} \times_{LM} \inhom{\cech{V_2}}{X} 
		& \simeq 
		\inhom{\cech{V_1}}{X} \times_{ \inhom{\cech{V_{12}}}{X}} 
		\left(
		\inhom{\cech{V_{12}}}{X} \times_{LM} \inhom{\cech{V_{12}}}{X}
		\right)
		\times_{ \inhom{\cech{V_{12}}}{X}} \inhom{\cech{V_2}}{X} \\
		&\simeq 
		\inhom{\cech{V_1}}{X} \times_{ \inhom{\cech{V_{12}}}{X}} 
		\inhom{\cech{V_{12}}}{(X\times_M X)} 
		\times_{ \inhom{\cech{V_{12}}}{X}} \inhom{\cech{V_2}}{X}
	\end{align*}
	which arise from the isomorphisms
	\[
		\inhom{\cech{V_{12}}}{X} \times_{LM} \inhom{\cech{V_{12}}}{X}
		\simeq
		\inhom{\cech{V_{12}}}{(X\times_M X)} 
	\]
	and $\inhom{\cech{V_{12}}}{\disc(M)} \simeq LM.$

	Now we have the following isomorph of (\ref{eq:looped_source_target}):
	\begin{equation}\label{eq:the_isomorph}
		\xymatrix{
			\inhom{\cech{V_1}}{X} \times_{ \inhom{\cech{V_{12}}}{X}} 
			\inhom{\cech{V_{12}}}{(X^\mathbbm{2})}
			\times_{ \inhom{\cech{V_{12}}}{X}} \inhom{\cech{V_2}}{X} 
			\ar[d]\\
			\inhom{\cech{V_1}}{X} \times_{ \inhom{\cech{V_{12}}}{X}} 
			\inhom{\cech{V_{12}}}{(X\times_M X)} 
			\times_{ \inhom{\cech{V_{12}}}{X}} \inhom{\cech{V_2}}{X}
		}
	\end{equation}
	which is the map induced from the map
	\begin{equation}\label{eq:map_induced_on_internal_hom_by_source_target}
		\inhom{\cech{V_{12}}}{(X^\mathbbm{2})} \longrightarrow
		\inhom{\cech{V_{12}}}{(X\times_M X)} 
	\end{equation}
	by interated pullback. 
	This is, in turn, induced by applying the functor $\inhom{\cech{V_{12}}}{(-)} $ 
	to the internal functor
	\[
		(S,T)\colon X^\mathbbm{2} \to X\times_M X,
	\]
	which is submersive by Lemma~\ref{lemma:some_standard_submersive_functors}.
	We can then apply Lemma~\ref{lemma:submersive_functor_gives_submersion}.1 to see that the map (\ref{eq:map_induced_on_internal_hom_by_source_target}) is a submersion.

	Now notice that we can apply Lemma~\ref{lemma:submersive_functor_gives_submersion}.2 to the maps 
	$\inhom{\cech{V_i}}{X} \to \inhom{\cech{V_{12}}}{X}$ 
	($i=1,2$) to see they are submersions. 
	It also follows from Lemma~\ref{lemma:some_standard_submersive_functors} together with Lemma~\ref{lemma:submersive_functor_gives_submersion} that the two maps 
	$\inhom{\cech{V_{12}}}{(X^\mathbbm{2})} \to \inhom{\cech{V_{12}}}{X}$
	induced by $S,T\colon X^\mathbbm{2} \to X$,	and the two maps 
	$
	\inhom{\cech{V_{12}}}{(X\times_M X)} \to \inhom{\cech{V_{12}}}{X}$ 
	induced by the two projections are submersions.
	Now we can apply Lemma~\ref{lemma:map_of_wide_pullbacks}, as the diagram giving the iterated pullback defining the map (\ref{eq:the_isomorph}) to get the desired result, namely that (\ref{eq:the_isomorph}) is a submersion.
\end{proof}

Let us say a gerbe has \emph{connected stabilisers} if $\Lambda X \to X_0$ is a bundle of connected groups.
It then follows that $X_1 \to X_0\times_M X_0$ has connected fibres, and the group $\Lambda X_x$ acts simply transitively on all fibres $(s,t)^{-1}(x,y)$.

Call a submersion $E\to B$ \emph{curvewise trivial} if for every map $\eta\colon [a,b]\to B$, the projection $\eta^*E \to [a,b]$ is isomorphic to a trivial bundle $[a,b]\times F \to [a,b]$.
For a gerbe $X$ that has connected stabilisers, if $(s,t)$ is curvewise trivial then the manifold $F$ is connected.

A gerbe $X$ that has $(s,t)$ curvewise trivial satisfies the property that a lift, as shown in the diagram
\[
	\xymatrix{
		[a,c] \ar[r] \ar[d] & X_1 \ar[d]^{(s,t)}\\
		[a,b] \ar[r] \ar@{-->}[ur] \ar[r] & X_0\times_M X_0 
	}
\]
always exists, for $c\in [a,b)$.
If the gerbe additionally has connected stabilisers, then there is always a lift as in this diagram:
\[
	\xymatrix{
		[a,c]\coprod [d,b] \ar[r] \ar[d] & X_1 \ar[d]^{(s,t)}\\
		[a,b] \ar[r] \ar@{-->}[ur] \ar[r] & X_0\times_M X_0 
	}
\]
for $a < c < d < b$.
Both of these follow from the ability to extend functions $[a,c] \to F$ (respectively $[a,c]\coprod [d,b] \to F$) to $[a,b]$, using Corollary~\ref{Corollary_to_Seely} (and the fact $F$ is connected in the latter case).

\begin{lemma}\label{lemma:(s,t)_surjective_subm}

	Let $M$ be a finite-dimensional manifold and $X$ a Lie groupoid that is a gerbe on $M$ with $(s,t)$ curvewise trivial, then $\LL X_0 \to LM$ is a surjective submersion. 
	If additionally $X$ has connected stabilisers then $(s,t)^{\LL X}\colon \LL X_1 \to \LL X_0\times_{LM} \LL X_0$ is a surjective submersion.

\end{lemma}

\begin{proof} 

	For the first statement, note that it is immediate that $\LL \vC(X_0) \to LM$ is surjective, because since $X_0\to M$ is a surjective submersion, it has local sections which can be used to lift locally any loop $\gamma\colon S^1 \to M$.
	Then to show that $\LL X \to \LL \vC(X_0)$ is surjective, we need to use the first assumption on $(s,t)$.

	Note that we only need to show we can lift paths $[a,b] \to X_0\times_M X_0$ through $(s,t)\colon X_1 \to X_0\times_m X_0$, where $[a,b] \subset \vC(V)_1$; there are no compatibility conditions.
	But note that since $X_1$ trivialises after pulling back to $[a,b]$, one can just use a section to lift paths as needed. 
	Thus $\LL X \to \LL \vC(X_0)$ is surjective, and so the first claim follows.

	For the second claim we only need to prove that $(s,t)^{\LL X}$ is surjective (it is already a submersion), so consider a single component $\inhom{\vC(V_1)}{X} \times_{LM} \inhom{\vC(V_2)}{X} \subset \LL X_0\times_{LM} \LL X_0$. 
	It suffices to prove that $\inhom{\vC(V_{12})}{(X^\mathbbm{2})} \to \inhom{\vC(V_{12})}{X} \times_{LM} \inhom{\vC(V_{12})}{X} \simeq \inhom{\vC(V_{12})}{(X\times_M X)}$ is surjective, since $(s,t)^{\LL X}$ is a disjoint union of pullbacks of such maps.
	Write $V = V_{12}$, and consider $\gamma = (\gamma_1,\gamma_2) \colon \vC(V) \to X\times_M X$.
	We need to find a lift $\widehat{\gamma}$ as in the diagram:
	\[
		\xymatrix{
		& X^\mathbbm{2} \ar[d] \\
		\vC(V) \ar[r]_-\gamma \ar@{-->}[ur]^{\widehat{\gamma}} & X\times_M X
		}
	\]
	We will iterate through the connected components of $V$ to define $\widehat{\gamma}$ on both objects and arrows.
	The functor $\gamma$ has an underlying object component a map $\coprod_{i=0}^n J_i \to X_0^{[2]}$, and starting with  $J_0 = [a_0,b_0]$ we can find an arrow $b\colon \gamma_1(a_0) \to \gamma_2(a_0) \in X_1$.
	This uses the fact $X_1 \to X_0^{[2]}$ is surjective.
	Since $(s,t)$ is curvewise trivial, we can find a section over $J_0$, and hence a map $J_0 \to X_1 = (X^\mathbbm{2})_0$.

	If we denote $J_{i-1}\cap J_i$ by $J^{i-1}_i$ (working mod $n+1$), then by naturality the object component $\widehat{\gamma}_0$ of the lift $\gamma$ on $J^{i-1}_i \subset J_i$ is determined by its value on $J^{i-1}_i \subset J_{i-1}$.
	By this we mean that for $\widehat{\gamma}$ to be a functor to $X^\mathbbm{2}$, or in other words a natural transformation $\gamma_1 \Rightarrow \gamma_2$, it must for every point in $J^{i-1}_i\subset \vC(V)_1$ satistfy naturality. 
	Thus from the lift on $J_0$ we can define the lift on $J^0_1 \subset J_1$, and then again use the fact $(s,t)$ is curvewise trivial to continue the lift on the rest of $J_1$.

	So starting from $J_0$ we can work through the indexing set for the cover until we have defined $\widehat{\gamma}_0$ on $J_{n-1}$, and hence on $J^{n-1}_n \subset J_n$. 
	The first lift, on $J_0$ defines $\widehat{\gamma}_0$ on $J^n_0 \subset J_n$, and so we need to be able to define a map $J_n \to X_1$ extending both of these partial maps.
	This is the situation as in Figure~\ref{fig:circle_diag}, where we need to define the dotted portion of upper central arc.

\begin{figure}\label{fig:circle_diag}
\centering
\begin{tikzpicture}[scale=1.3]
\pgfmathsetmacro{\Router}{10}
\pgfmathsetmacro{\Rinner}{7.5}
\pgfmathsetmacro{\smallgap}{0.15}
\pgfmathsetmacro{\arrowlength}{\Router-\Rinner - 2*\smallgap}

\centerarc[gray,ultra thick,dash pattern=on 3pt off 3pt](0,0)(39:43:\Router);
\centerarc[black,ultra thick](0,0)(43:52.5:\Router);
\draw[->] (47.5:\Rinner+\smallgap) -- +(47.5:\arrowlength);
\draw[->] (50:\Rinner+\smallgap) -- +(50:\arrowlength);
\draw[->] (52.5:\Rinner+\smallgap) -- +(52.5:\arrowlength);

\centerarc[black,ultra thick](0,0)(47.5:65:\Rinner) node[midway, below left=-1mm] {$\widehat{\gamma}_0(J_0)  $}; 
\draw[<-] (57.5:\Rinner+\smallgap) -- +(57.5:\arrowlength);
\draw[<-] (60:\Rinner+\smallgap) -- +(60:\arrowlength);
\draw[<-] (65:\Rinner+\smallgap) -- +(65:\arrowlength);
\draw[<-] (62.5:\Rinner+\smallgap) -- 
node[fill=white,rounded corners=1ex,inner sep=0.5mm]{$\widehat{\gamma}_1(J_0^n)$}
+(62.5:\arrowlength);

\centerarc[black,ultra thick](0,0)(57.5:65:\Router)node[above right=-0.5mm]{$\widehat{\gamma}_0(J_0^n)$};
\centerarc[black,ultra thick,dotted](0,0)(65:71:\Router);
\centerarc[black,ultra thick](0,0)(71:81:\Router)node[above right]{$\widehat{\gamma}_0(J_n^{n-1})$};

\draw[->] (71:\Rinner+\smallgap) -- +(71:\arrowlength);
\draw[->] (73.5:\Rinner+\smallgap) -- +(73.5:\arrowlength);
\draw[->] (78.5:\Rinner+\smallgap) -- +(78.5:\arrowlength);
\draw[->] (81:\Rinner+\smallgap) -- +(81:\arrowlength);
\draw[->] (76:\Rinner+\smallgap) -- 
node[fill=white,rounded corners=2ex,inner sep=0.5mm]{$\widehat{\gamma}_1(J_n^{n-1})$}
+(76:\arrowlength);

\centerarc[black,ultra thick](0,0)(71:90:\Rinner) node[below=1.5mm, midway] {$\widehat{\gamma}_0(J_{n-1})$};
\draw[<-] (85:\Rinner+\smallgap) -- +(85:\arrowlength);
\draw[<-] (87.5:\Rinner+\smallgap) -- +(87.5:\arrowlength);
\draw[<-] (90:\Rinner+\smallgap) -- +(90:\arrowlength);

\centerarc[black,ultra thick](0,0)(85:93:\Router);
\centerarc[gray,ultra thick,dashed](0,0)(93:97:\Router);
\end{tikzpicture}
\caption{Defining a lift to $X^\mathbbm{2}$}
\end{figure}
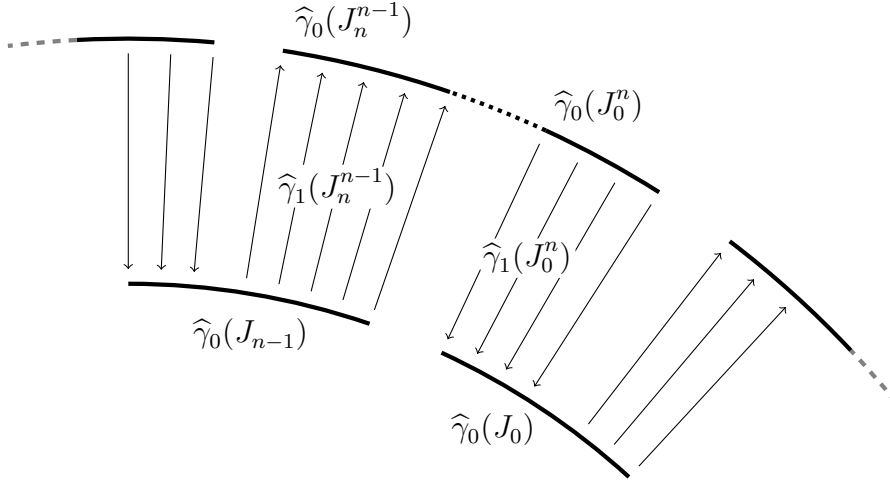

	It is here we use the hypothesis that $X$ has connected stabilisers, since if we pull back $X_1 \to X_0^{[2]}$ along $\gamma\big|_{J_n}$, we can trivialise to $J_n \times A$.
	Then $A$ is necessarily a connected Lie group, so we can extend the map $J_n \supset J^{n-1}_n \coprod J^n_0 \to A$ to all of $J_n$, completing the lift $\widehat{\gamma}\colon \vC(V) \to X^\mathbbm{2}$, and the proof.
\end{proof}

Thus we get the first main result of this section.

\begin{theorem}\label{theorem:loop_stack_is_a_gerbe}

	For a differentiable gerbe $\cX$ presented by a Lie groupoid $X$ with connected stabilisers such that $(s,t)$ is curvewise trivial, then $\cHom(S^1,\cX)$ is a Fr\'echet differentiable gerbe.
	
\end{theorem}

We have the following additional result if we know some more about the gerbe $X$.

\begin{proposition}\label{proposition:Auts_in_looped_bundle_gerbe}

	Let $X \to M$ be an abelian $\mathcal{A}$-gerbe where $\mathcal{A}$ is a locally trivial bundle of connected (abelian Lie) groups. 
	Then $\Lambda\LL X \simeq \LL X_0 \times_{LM} L\mathcal{A}$. 

\end{proposition}

\begin{proof} 
The assumptions on $X$ mean that $\LL X$ is a gerbe.
The definition of $\Lambda \LL X$ is that it is the pullback
\[
	\xymatrix{
	\Lambda \LL X \ar[r] \ar[d] & \LL X_1 \ar[d]^-{(s,t)^{\LL X}} \\
	\LL X_0 \ar[r]_-\Delta & (\LL X_0)^{[2]}
	}
\]
and on the component $\inhom{\vC(V)}{X} \subset \LL X_0$ this is precisely the pullback
\[
	\xymatrix{
		(\Lambda \LL X)_V \ar[r] \ar[d] & \inhom{\vC(V)}{(X^\mathbbm{2})} \ar[d] & \\
		\inhom{\vC(V)}{X} \ar[r] & \inhom{\vC(V)}{X} \times_{LM} \inhom{\vC(V)}{X} \ar@{}[r]|-\simeq & \inhom{\vC(V)}{(X\times_M X)}\,.
	}
\]
But since $\inhom{\vC(V)}{(-)}$ preserves strict pullbacks of submersive functors, we have $(\Lambda \LL X)_V \simeq \inhom{\vC(V)}{(X\times_{X\times_M X}X^\mathbbm{2})}$.
By Lemma~\ref{lemma:inertia_bundle_pullback_of_arrow_groupoid}, $X\times_{X\times_M X}X^\mathbbm{2} \simeq \Lambda X/\!/ X$, and since $X$ presents an abelian $\mathcal{A}$-gerbe, $\Lambda X/\!/ X \simeq \mathcal{A}\times_M X$, by Lemma~\ref{lemma:inertia_with_adj_action_pullback_from_base}. 
Thus the summand $(\Lambda \LL X)_V$ of $\Lambda\LL X$ over $\inhom{\vC(V)}{X}$ is isomorphic to $\inhom{\vC(V)}{(\mathcal{A}\times_M X)}\simeq L\mathcal{A}\times_{LM} \inhom{\vC(V)}{X}$ (where we have implicitly identified $\inhom{\vC(V)}{M}$ with $LM$ and similarly for $L\mathcal{A}$).
\end{proof}

This gives us the final main result, and in fact the original motivation for this paper.

\begin{theorem}\label{theorem:loop_gerbe}

	Let $M$ be a finite-dimensional smooth manifold, $\mathcal{A}$ be a locally trivial bundle of connected abelian Lie groups on $M$ and $X \to M$ a finite-dimensional abelian $\mathcal{A}$-gerbe. 
	Then $\LL X$ is an abelian $L\mathcal{A}$-gerbe on $LM$.

\end{theorem}

\begin{proof} 
Theorem~\ref{theorem:loop_stack_is_a_gerbe} ensures that $\LL X$ is again a gerbe, and from Proposition~\ref{proposition:Auts_in_looped_bundle_gerbe} we know that $\Lambda \LL X$ descends to a bundle of groups $L\mathcal{A} \to LM$.
Hence we know $\LL X$ is an $L\mathcal{A}$-gerbe, and we thus need to show that $\LL X$ is an \emph{abelian} gerbe.
This will be done by showing condition 2 of Lemma~\ref{lemma:LR_actions_agree} holds for $\LL X$, given that condition 3 of Lemma~\ref{lemma:LR_actions_agree} holds for $X$.
That is, we need to show the diagram
\[
	\xymatrix{
		\Lambda \LL X \times_{\LL X_0} \LL X_1\ar[rr] \ar[d]_{\id_{\Lambda \LL X}\times(s,t)^{\LL X}} && \Lambda \LL X \ar@{=}[d]\\
		\Lambda \LL X \times_{\LL X_0} \LL X_0\times_{LM}\LL X_0 \ar[rr] && \Lambda \LL X
	}
\]
commutes.
This reduces (using Lemma~\ref{lemma:inertia_bundle_pullback_of_arrow_groupoid}) to showing the following diagram commutes, for all $V_1,V_2$:
\[
	\centerline{\xymatrix{
		\inhom{\vC(V_1)}{(\Lambda X/\!/X)} \times_{\inhom{\vC(V_{12})}{X}} \inhom{\vC(V_{12})}{(X^\mathbbm{2})} \times_{\inhom{\vC(V_{12})}{X}} \inhom{\vC(V_2)}{X}
		\ar[rr]^-{\act_\LL} \ar[d]_{\pr_{124}} \ar@{}[drr]|{?} && 
		\inhom{\vC(V_2)}{(\Lambda X/\!/X)} \ar@{=}[d]\\
		\inhom{\vC(V_1)}{(\Lambda X/\!/X)} \times_{\inhom{\vC(V_1)}{X}} 
		\left(\inhom{\vC(V_1)}{X} \times_{LM} \inhom{\vC(V_2)}{X}\right)
		\ar[rr] && \inhom{\vC(V_2)}{(\Lambda X/\!/X)}
	}}
\]
Using the isomorphism $\inhom{\vC(V)}{(\Lambda X/\!/X)}  \simeq \inhom{\vC(V)}{(\mathcal{A} \times_M X)} \simeq L\mathcal{A} \times_{LM} \inhom{\vC(V)}{X}$, we can rewrite the desired diagram as
\[\centerline{
	\xymatrix{
		L\mathcal{A} \times_{LM} \inhom{\vC(V_1)}{X} 
		\times_{\inhom{\vC(V_{12})}{X}}
		\inhom{\vC(V_{12})}{(X^\mathbbm{2})}
		\times_{\inhom{\vC(V_{12})}{X}} \inhom{\vC(V_2)}{X}
		\ar[rr]^-{\pr_{14}} \ar[d]_{\pr_{124}} && 
		L\mathcal{A} \times_{LM} \inhom{\vC(V_2)}{X} \ar@{=}[d]\\
		L\mathcal{A} \times_{LM} \inhom{\vC(V_1)}{X} 
		\times_{LM} \inhom{\vC(V_2)}{X}
		\ar[rr]_-{\pr_{13}} && L\mathcal{A} \times_{LM} \inhom{\vC(V_2)}{X}
	}
	}
\]
in other words, we need to prove that the action map $\act_\LL$ above (defined using conjugation of transformations of anafunctors) is, up to isomorphism, the projection $\pr_{14}$ in the top row of the preceeding diagram.

Now note that condition 2 in Lemma~\ref{lemma:LR_actions_agree} for $X$ (which holds since we are assuming $X$ is an abelian $\mathcal{A}$-gerbe) can be rewritten as
\[
	\xymatrix{
		\mathcal{A} \times_M X^\mathbbm{2} \ar[r]^-{\id\times T}  & \mathcal{A} \times_M X \ar[d]^\simeq\\
		\Lambda X /\!/ X \times_X X^\mathbbm{2} \ar[r]^-\act \ar[d]_{\id\times(S,T)} \ar[u]^\simeq & \Lambda X /\!/ X \ar@{=}[d] \\
		\Lambda X /\!/ X \times_X (X\times_M X) \ar[r] \ar[d]_\simeq & \Lambda X /\!/ X  \\
		\mathcal{A} \times_M (X \times_M X) \ar[r]_-{\pr_{13}} & \mathcal{A} \times_M X  \ar[u]_\simeq
	}
\]
In other words, the action of $X^\mathbbm{2}$ on $\Lambda X/\!/ X$ is, up to isomorphism, essentially given by the target functor $T\colon X^\mathbbm{2} \to X$.
We will in particular use the top square of this diagram for the next step of the proof.

The map $\act_\LL$ is defined (using the incorporated simplifications) to be the composite of the left column of arrows in the diagram 
\[
\centerline{
	\xymatrix@C=2ex{
		\inhom{\vC(V_1)}{(\Lambda X/\!/X)} \times_{\inhom{\vC(V_{12})}{X}} \inhom{\vC(V_{12})}{(X^\mathbbm{2})} \times_{\inhom{\vC(V_{12})}{X}} \inhom{\vC(V_2)}{X}
		\ar[d] \ar[r]^-\simeq & 
		L\mathcal{A} \times_{LM} \inhom{\vC(V_1)}{X} \times_{\inhom{\vC(V_{12})}{X}} \inhom{\vC(V_{12})}{(X^\mathbbm{2})} \times_{\inhom{\vC(V_{12})}{X}} \inhom{\vC(V_2)}{X} \ar[d]\\
		\inhom{\vC(V_{12})}{(\Lambda X/\!/X)} \times_{\inhom{\vC(V_{12})}{X}} \inhom{\vC(V_{12})}{(X^\mathbbm{2})} \times_{\inhom{\vC(V_{12})}{X}} \inhom{\vC(V_2)}{X} \ar[d]_-\simeq \ar[r]^-\simeq
		& L\mathcal{A} \times_{LM} \inhom{\vC(V_{12})}{(X^\mathbbm{2})} \times_{\inhom{\vC(V_{12})}{X}} \inhom{\vC(V_2)}{X}  \ar[d]^-\simeq\\
		\inhom{\vC(V_{12})}{\left(\Lambda X/\!/X\times_X X^\mathbbm{2} \right)}
		\times_{\inhom{\vC(V_{12})}{X}} \inhom{\vC(V_2)}{X} \ar[r]^-\simeq \ar[d]_{\inhom{\vC(V_{12})}{\act}\times \id}
		&
		\inhom{\vC(V_{12})}{\left(\mathcal{A} \times_M X^\mathbbm{2}\right)}
		\times_{\inhom{\vC(V_{12})}{X}} \inhom{\vC(V_2)}{X} 
		\ar[d]^{\inhom{\vC(V_{12})}{(\id\times T)}\times \id}\\
		\inhom{\vC(V_{12})}{(\Lambda X/\!/X)}\times_{\inhom{\vC(V_{12})}{X}} \inhom{\vC(V_2)}{X} \ar[r]^-\simeq \ar[d]_-\simeq &
		\inhom{\vC(V_{12})}{\left(\mathcal{A} \times_M X\right)}
		\times_{\inhom{\vC(V_{12})}{X}} \inhom{\vC(V_2)}{X} \ar[d]^-\simeq\\
		\inhom{\vC(V_2)}{(\Lambda X/\!/X)} \ar[r]^-\simeq & L\mathcal{A} \times_{LM} \inhom{\vC(V_2)}{X}
	}
}
\]
and the square second from the bottom commutes because of the assumption that $X$ is an abelian $\mathcal{A}$-gerbe. The composite of the right column of arrows is just $\pr_{14}$, and hence condition 2 of Lemma~\ref{lemma:LR_actions_agree} holds, and so $\LL X$ is an abelian $L\mathcal{A}$-gerbe, as we needed to prove.
\end{proof}

\begin{corollary}

	If $A$ is a connected abelian Lie group and $X$ is an $A$-bundle gerbe on $M$, then $\LL X$ is an $LA$-bundle gerbe.

\end{corollary}

\begin{proof}

	An $A$-bundle gerbe $X\to M$ is an abelian $A\times M$-gerbe and $(s,t)$ is the projection for a locally trivial bundle, so $\LL X$ is an abelian $L(A\times M) \simeq LA \times LM$ gerbe.
	Thus $\LL X$ is an $LA$-bundle gerbe.
\end{proof}

\begin{remark}

	We would like to apply this result to the basic gerbe on a Lie group, since then we get a gerbe over the free loop group that is \emph{multiplicative}. 
	This is fine if we use one of the finite-dimensional models, but it would be useful if we could also use the infinite-dimensional strict model $\String_G^{\text{BCSS}}$ described in \cite{BCSS_07}.
	The results from Section \ref{sec:diffeological_presentation} show that $\LL \String_G^{\text{BCSS}}$ is at worst a diffeological groupoid.
	Since $\LL$ preserves products up to equivalence this in fact a coherent diffeological 2-group (see eg~\cite{Baez_Lauda_04}).
	We conjecture, based on private discussion with Alexander Schmeding, that the results of this paper should apply to $\String_G^{\text{BCSS}}$, and in fact Fr\'echet--Lie groupoids with (adapted) local additions and possibly also smoothly locally regular\footnote{See \cite[Definition 3.7]{Stacey_13} for the formal definition: a smoothly locally regular map is one that is equipped with the analogue of a local addition, smoothly parameterising submersion charts in the domain.} source and target more generally.

\end{remark}
\providecommand{\bysame}{\leavevmode\hbox to3em{\hrulefill}\thinspace}
\providecommand{\MR}{\relax\ifhmode\unskip\space\fi MR }
\providecommand{\MRhref}[2]{%
  \href{http://www.ams.org/mathscinet-getitem?mr=#1}{#2}
}
\providecommand{\href}[2]{#2}


\vspace{5mm}
\noindent David Michael Roberts, \href{mailto:david.roberts@adelaide.edu.au}{\texttt{david.roberts@adelaide.edu.au}} \\

\noindent Raymond F.\ Vozzo, \href{mailto:raymond.vozzo@adelaide.edu.au}{\texttt{raymond.vozzo@adelaide.edu.au}} \\

\noindent School of Mathematical Sciences, University of Adelaide \\
Adelaide 5005, Australia \\


\end{document}